\documentclass[11pt,a4paper]{article}
\usepackage[margin=0.88in]{geometry}
\usepackage[T1]{fontenc}
\usepackage[utf8]{inputenc}
\usepackage{lmodern}
\usepackage{amsmath,amssymb,amsthm,mathtools,bm,mathrsfs}
\usepackage{microtype,xurl,booktabs,array,tabularx,enumitem}
\usepackage{caption}
\usepackage{needspace}
\usepackage{xcolor,float,tikz}
\usepackage{cite}
\definecolor{linkblue}{RGB}{26,52,108}
\usepackage[colorlinks=true,linkcolor=linkblue,citecolor=linkblue,urlcolor=linkblue]{hyperref}
\usepackage[nameinlink,noabbrev]{cleveref}
\usetikzlibrary{arrows.meta}
\hypersetup{pdftitle={Matching Randomized Bounds for Finite-Sum Nonconvex-Concave Optimization},pdfauthor={},pdfsubject={Individual and averaged smoothness; bounded dual domains; Moreau stationarity},pdfkeywords={finite-sum optimization, nonconvex-concave minimax optimization, oracle complexity, randomized lower bounds, Moreau stationarity}}
\newtheorem{theorem}{Theorem}[section]
\newtheorem{lemma}[theorem]{Lemma}
\newtheorem{proposition}[theorem]{Proposition}
\newtheorem{corollary}[theorem]{Corollary}
\theoremstyle{definition}
\newtheorem{definition}[theorem]{Definition}
\theoremstyle{remark}
\newtheorem{remark}[theorem]{Remark}
\newcommand{\R}{\mathbb R}
\newcommand{\E}{\mathbb E}
\newcommand{\Prob}{\mathbb P}
\newcommand{\PP}{\mathbb P}
\newcommand{\eps}{\varepsilon}
\newcommand{\norm}[1]{\left\lVert#1\right\rVert}
\newcommand{\ip}[2]{\left\langle#1,#2\right\rangle}
\newcommand{\op}{\mathrm{op}}
\newcommand{\argmin}{\operatorname*{argmin}}
\newcommand{\barB}{\overline B}

\newcommand{\Span}{\operatorname{span}}

\newcommand{\calH}{\mathcal H}

\newcommand{\T}{\mathsf T_{\mathrm{IS}}}
\newcommand{\Tas}{\mathsf T_{\mathrm{AS}}}
\newcommand{\rate}{\mathcal F}

\DeclareMathOperator{\diam}{diam}
\floatstyle{ruled}
\newfloat{algorithm}{tbp}{loa}
\floatname{algorithm}{Algorithm}
\crefname{algorithm}{Algorithm}{Algorithms}
\setlist[itemize]{leftmargin=1.4em,itemsep=3pt,topsep=4pt}
\setlist[enumerate]{leftmargin=1.8em,itemsep=3pt,topsep=4pt}
\allowdisplaybreaks[1]
\title{\textbf{Matching Randomized Bounds for Finite-Sum Nonconvex--Concave Optimization}}
\newcommand{\authornames}{Wendao Wu, Haihan Zhang, Qian Chen, Chenheng Zhang, Chunyuan Zheng, Cong Fang, Haoxuan Li, Zhouchen Lin}
\newcommand{\affiliation}{Peking University}

\newcommand{\mail}[1]{\href{mailto:#1}{\texttt{#1}}}
\newif\ifanonymous
\anonymousfalse
\ifanonymous
\author{Anonymous Authors}
\hypersetup{pdfauthor={Anonymous Authors}}
\else
\author{Wendao Wu$^{1,*}$\quad Haihan Zhang$^{1,*}$\quad Qian Chen$^{1,*}$\\[0.20em]
Chenheng Zhang$^{1,*}$\quad Chunyuan Zheng$^1$\\[0.20em]
Cong Fang$^{1,\dagger}$\quad Haoxuan Li$^{1,\dagger}$\quad Zhouchen Lin$^{1,\dagger}$\\[0.60em]
\small $^1$\affiliation\\[0.35em]
\footnotesize\mail{colombo.2026@outlook.com}\\[0.35em]
\small $^*$Equal contribution.\quad $^\dagger$Corresponding authors.}
\hypersetup{pdfauthor={\authornames}}
\fi
\date{}

\newcommand{\ResearchAgentSystem}{Colombo, our laboratory's internal auto-research system}

\begin{document}
\maketitle
\begin{abstract}
We characterize the randomized component-oracle complexity of finite-sum
nonconvex--concave optimization with a bounded dual domain. For
$\Phi(x)=\max_{y\in Y}n^{-1}\sum_{i=1}^n f_i(x,y)$, the goal is to find
an $\varepsilon$-stationary point of the Moreau envelope
$\Phi_{1/(2L)}$ with probability at least $2/3$ and a fixed query budget.
Let $A=L\Delta/\varepsilon^2$, $B=LD_Y/\varepsilon$, and
$\bar B=\max\{1,B\}$, where $\Delta$ bounds the initial primal gap and $D_Y$ bounds the dual diameter.
For $A$ above a universal constant, the worst-case complexity is
\[
\widetilde\Theta\!\left(n+A\min\left\{
\sqrt n\,\bar B+\bar B^4,\ n^{3/4}\bar B\right\}\right).
\]
The bounds agree within constants for $B\le n^{1/4}$ and within one
logarithm for larger $B$. Upper bounds require only mean-squared
averaged smoothness and concavity of the average in the dual variable;
lower bounds already hold under individual smoothness. The lower
constructions cover adaptive component selection, arbitrary query
norms, private randomization, and outputs that were never queried.
A cone construction gives the small-diameter bound, while a serial
branch construction and balanced replication give the remaining two
regimes. On the upper side, a persistent estimator with primal
averaging complements accelerated proximal dual steps equipped with
an explicit primal certificate. The characterization takes the worst
case over admissible known dual domains and finite dimensions.
\end{abstract}
\noindent\textbf{Keywords:} finite-sum optimization; nonconvex--concave minimax optimization; oracle complexity; randomized lower bounds; Moreau stationarity.

\par\medskip
{\small
\begin{samepage}
\noindent\textbf{AI Usage.}
Nearly the entire research pipeline for this paper was carried out by
\ResearchAgentSystem{}, powered by GPT-5.6 Sol. The system also conducted a
Lean-backed article audit of the resulting manuscript. The authors subsequently reviewed and approved the
mathematical claims, presentation, and formal artifacts, and take
responsibility for the final manuscript. The complete Lean audit report and the system's technical report will be made public at a later date.

\paragraph{Publication Intent and Recognition of Independent Work.}
This preprint is released as a research output of our study of
automated scientific research, for public scrutiny and discussion.
As we do not intend to submit this manuscript to a journal or
conference, it has not undergone the full editorial refinement
and internal review we would normally require for submission.
The authors remain responsible for its contents and for correcting
any errors or inaccurate attribution.

We recognize that releasing these results may affect researchers
who have independently been pursuing the same questions, including
work that is not yet public.
If this applies to your work, we invite you to contact us at
\texttt{colombo.2026@outlook.com}.
A brief description of the overlap, together with any dated
materials you are comfortable sharing, would help us understand
its scope and chronology.
Such materials might include draft excerpts, presentation slides,
or revision records; a complete manuscript or unpublished proof
is not required to begin the conversation.

We will review these cases carefully and update citations,
acknowledgments, and accounts of independent or concurrent work
as appropriate, including revising priority claims where warranted.
Nonpublic materials will be treated confidentially and used only
to assess the overlap unless their authors authorize otherwise.
We will obtain explicit permission from the authors before
publicly discussing such materials or providing them to an
automated research system.
\par
\end{samepage}
}

\section{Introduction}
\label{sec:introduction}
Consider the finite-sum minimax problem
\begin{equation}\label{eq:problem-main}
 \min_{x\in\R^{d_x}}\Phi(x),\qquad
 \Phi(x)=\max_{y\in Y}f(x,y),\qquad
 f(x,y)=\frac1n\sum_{i=1}^n f_i(x,y),
\end{equation}
where the average is smooth, possibly nonconvex in $x$, and concave
in $y$. The compact convex set $Y$ gives a finite diameter, but it
does not supply positive dual curvature. Consequently, the value
function $\Phi$ can be nonsmooth even when every component is smooth.
Moreau-envelope stationarity provides a common target for this
setting and underlies proximal approaches to nonconvex--concave
optimization~\cite{thekumparampil,lin2020,catalyst}.

Two resources interact in this problem. A large finite sum makes
full-gradient evaluations expensive, while a large dual domain makes
the maximization subproblem harder to localize. Classical
variance-reduction bounds for smooth minimization identify the
$\sqrt n$ scale~\cite{spider,page,peng}. For minimax optimization,
Catalyst-SVRG already yields an $n^{3/4}\varepsilon^{-3}$ polynomial
upper rate in the high-accuracy regime~\cite[Corollary~4.3]{catalyst}.
These facts do not determine how the optimal complexity changes
between a negligible dual domain and a large one, or whether that
complexity can be forced against randomized algorithms making
arbitrary component queries.

We resolve this joint dependence for a precise exact-oracle model.
With
\begin{equation}\label{eq:dimensionless-main}
 A=\frac{L\Delta}{\varepsilon^2},\qquad
 B=\frac{LD_Y}{\varepsilon},\qquad \bar B=\max\{1,B\},
\end{equation}
the accuracy-dependent coefficient has three regimes:
\[
 \sqrt n\,\bar B\quad(B\le n^{1/6}),\qquad
 B^4\quad(n^{1/6}\le B\le n^{1/4}),\qquad
 n^{3/4}B\quad(B\ge n^{1/4}).
\]
Here $D_Y$ is the permitted dual diameter, and the additive $n$ term
is retained in every regime. In particular, the intermediate regime
is not obtained by multiplying the smooth minimization complexity by
a single diameter factor. Its $B^4$ term is enforced by a hard
instance with an effective component count of order $B^4$, replicated
into the prescribed finite sum.

\paragraph{Contributions.}
Our results have four parts.
\begin{itemize}
\item \textbf{A uniform characterization.} We prove matching bounds
in $n$, $L$, $\Delta$, $D_Y$, and $\varepsilon$, with a single
logarithmic gap only in the large-diameter regime. The lower bounds
and upper bounds use the same fixed-budget, constant-success
criterion and the same Moreau parameter.
\item \textbf{Lower bounds for unrestricted randomized queries.}
The component oracle returns both the exact value and the full joint
gradient. Our constructions tolerate adaptive indices, dense
arbitrary-norm queries, and arbitrary final outputs. They do not
assume that queries lie in a revealed span or support.
\item \textbf{The same rates under two smoothness assumptions.}
The hard instances have individually $L$-Lipschitz gradients. Both
algorithms need only a mean-squared bound on component gradient
differences, with the same numerical constant $L$. Thus individual
and averaged smoothness have identical minimax rates in this model,
up to the stated logarithm.
\item \textbf{Explicit proof interfaces.} The lower proofs verify
exact value-and-gradient simulation and account for source releases
chronologically. The upper proofs handle biased anchor estimates,
constrained dual optima, and a computable primal certificate. All
construction formulas, parameter scalings, and algorithmic budgets
are included in the appendices.
\end{itemize}

The high-diameter polynomial upper rate is already present in
Catalyst~\cite{catalyst}; it is not itself a novelty claim of this
paper. Our characterization adds the full diameter crossover and
matching unrestricted randomized lower bounds, together with an
averaged-smoothness analysis and explicit fixed query budgets. The
lower bounds take a worst case over finite dimensions; the branch
construction can require exponentially large dimension. We state
this scope as part of the result.

\section{Problem Setting}
\label{sec:model-main}
Let $n\ge1$, $L,\Delta,\varepsilon>0$, and $D_Y\ge0$. The known set
$Y\subset\R^{d_y}$ is nonempty, compact, convex, and has diameter at
most $D_Y$. Component gradients are defined on a neighborhood of
$\R^{d_x}\times Y$; all smoothness inequalities and oracle queries
below are restricted to this product domain. The average
$f=n^{-1}\sum_i f_i$ is concave in $y$ for every $x$. We assume
\begin{equation}\label{eq:gap-main}
 \inf_x\Phi(x)>-\infty,\qquad
 \Phi(x^0)-\inf_x\Phi(x)\le\Delta.
\end{equation}
Only the average must be dual-concave. In particular, the individual
components need not be concave in $y$.
For the lower constructions, translating the primal coordinates
allows us to take $x^0=0$ without changing any class parameter.

\begin{definition}[Individual and averaged smoothness]
\label{def:smoothness-main}
Writing $z=(x,y)$, individual smoothness, denoted $\mathrm{IS}(L)$,
means
\begin{equation}\label{eq:is-main}
 \norm{\nabla f_i(z)-\nabla f_i(z')}\le L\norm{z-z'}
 \quad\text{for every }i\in[n]\text{ and }z,z'.
\end{equation}
Mean-squared averaged smoothness, denoted $\mathrm{AS}(L)$, means
\begin{equation}\label{as:eq:definition}
 \frac1n\sum_{i=1}^n\norm{\nabla f_i(z)-\nabla f_i(z')}^2
 \le L^2\norm{z-z'}^2\quad\text{for every }z,z'.
\end{equation}
\end{definition}
IS($L$) implies AS($L$). AS($L$) permits a component smoothness
constant as large as $\sqrt nL$, but it still makes the average
function $L$-smooth by Jensen's inequality. It is stronger than
smoothness of the average alone. For example, the components
$f_1(x)=Mx^2/2$ and $f_2(x)=(2L-M)x^2/2$ have an $L$-smooth average,
while their mean-squared smoothness is unbounded as $M/L$ grows.

\paragraph{Stationarity.}
The value function is $L$-weakly convex: $\Phi(x)+(L/2)\norm x^2$
is convex. Its Moreau envelope and proximal point are
\begin{equation}\label{eq:moreau-main}
 \Phi_{1/(2L)}(c)=\min_x\{\Phi(x)+L\norm{x-c}^2\},\qquad
 p(c)=\argmin_x\{\Phi(x)+L\norm{x-c}^2\}.
\end{equation}
The minimizer is unique and
$\nabla\Phi_{1/(2L)}(c)=2L(c-p(c))$.
The algorithm must return an output satisfying
\begin{equation}\label{eq:success-main}
 \Prob\!\left(\norm{\nabla\Phi_{1/(2L)}(\widehat x)}
                  \le\varepsilon\right)\ge\frac23.
\end{equation}
This certifies a nearby point $p(\widehat x)$ with a subgradient of
norm at most $\varepsilon$ and distance at most
$\varepsilon/(2L)$ from the output. No oracle for the Moreau
gradient or the proximal map is provided.

\paragraph{Oracle and complexity.}
A query specifies any $i\in[n]$ and $(x,y)\in\R^{d_x}\times Y$ and
returns
\[
 \bigl(f_i(x,y),\nabla_xf_i(x,y),\nabla_yf_i(x,y)\bigr).
\]
The algorithm may depend measurably on the complete transcript and
its private randomness. It may repeat indices, use arbitrary query
norms, and output an unqueried point. Computation and memory are
unrestricted; projection onto the known dual domain is not charged
as a component query. The budget is a deterministic upper bound on
the number of calls for every transcript and random seed.

Write $\T$ and $\Tas$ for the smallest worst-case budgets achieving
\eqref{eq:success-main} under IS($L$) and AS($L$), respectively.
The worst case includes admissible finite dimensions and known dual
domains, in addition to the component functions. Thus a lower bound
for diameter bound $D_Y$ asserts the existence of a hard known
domain, not hardness for every set with that diameter bound. The
construction domains are Euclidean balls. No exact initial dual
maximizer or uniform bound on component affine gradients is assumed.
The parameters in \eqref{eq:dimensionless-main} always use the same
$L$ as the class definition and Moreau envelope.

\section{Main Results}
\label{sec:results-main}
Define
\begin{equation}\label{eq:rate-main}
 \rate(n,B)=\min\{\sqrt n\,\bar B+\bar B^4,\ n^{3/4}\bar B\},
 \qquad \ell_{n,A,B}=\log(2+n+A+B).
\end{equation}

\begin{theorem}[Matching randomized component complexity]
\label{global:thm:main}\label{as:cor:integrated}
There are universal constants $A_*,c,C>0$ such that, for every
integer $n\ge1$, $L,\Delta,\varepsilon>0$, and $D_Y\ge0$ with
$A\ge A_*$,
\begin{equation}\label{global:eq:matching}
 c[n+A\rate(n,B)]\le\T\le\Tas
 \le C[n+A\rate(n,B)]\ell_{n,A,B}.
\end{equation}
More precisely, for either class,
\begin{equation}\label{global:eq:piecewise}
 \mathsf T=
 \begin{cases}
  \Theta(n+A\sqrt n\,\bar B),&0\le B\le n^{1/6},\\[2pt]
  \Theta(n+AB^4),&n^{1/6}\le B\le n^{1/4},\\[2pt]
  \widetilde\Theta(n+A n^{3/4}B),&B\ge n^{1/4}.
 \end{cases}
\end{equation}
Only the upper factor $\ell_{n,A,B}$ is suppressed in the last line;
the lower bound has no logarithmic loss.
\end{theorem}

The two transition points follow from
$\sqrt nB=B^4$ at $B=n^{1/6}$ and
$B^4=n^{3/4}B$ at $B=n^{1/4}$.
\Cref{fig:phase-main} depicts these exponents. For a fixed nonzero
dual diameter, sufficiently high accuracy eventually enters the
last regime and gives the familiar $\varepsilon^{-3}$ dependence.
The theorem also resolves the regimes in which $n$ grows jointly
with the target accuracy or the permitted diameter is small.

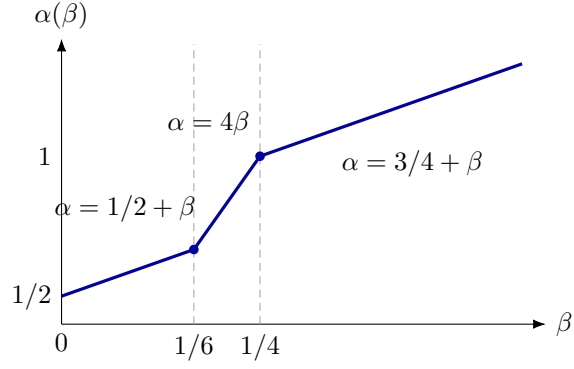
\begin{figure}[t]
\centering
\begin{tikzpicture}[x=10.5cm,y=3.7cm,>=Latex,font=\small]
 \draw[->] (0,0.4)--(0,1.42) node[above] {$\alpha(\beta)$};
 \draw[->] (0,0.4)--(.61,0.4) node[right] {$\beta$};
 \draw[densely dashed,black!35] (1/6,0.4)--(1/6,1.4);
 \draw[densely dashed,black!35] (.25,0.4)--(.25,1.4);
 \draw[very thick,blue!65!black] (0,.5)--(1/6,2/3);
 \draw[very thick,blue!65!black] (1/6,2/3)--(.25,1);
 \draw[very thick,blue!65!black] (.25,1)--(.58,1.33);
 \fill[blue!65!black] (1/6,2/3) circle (1.8pt);
 \fill[blue!65!black] (.25,1) circle (1.8pt);
 \node[below] at (0,.4) {$0$};
 \node[below] at (1/6,.4) {$1/6$};
 \node[below] at (.25,.4) {$1/4$};
 \node[left] at (0,.5) {$1/2$};
 \node[left] at (0,1) {$1$};
 \node[above,align=center] at (.08,.73) {$\alpha=1/2+\beta$};
 \node[above left] at (.25,1.04) {$\alpha=4\beta$};
 \node[below right] at (.34,1.06) {$\alpha=3/4+\beta$};
\end{tikzpicture}
\caption{The exponent of the accuracy-dependent coefficient:
for $B=n^\beta$, $\beta\ge0$, and growing $n$,
$\rate(n,B)=\Theta(n^{\alpha(\beta)})$.
The complete complexity includes the additive $n$ and multiplicative
$A$. The first two regimes match within constants; the third has
one logarithmic upper loss.}
\label{fig:phase-main}
\end{figure}

\begin{theorem}[Two upper bounds under averaged smoothness]
\label{thm:upper-main}
For all parameters in \cref{sec:model-main}, including $A<A_*$,
there are randomized algorithms with the following fixed budgets:
\begin{align}
 U_{\mathrm{avg}}&\le
 C\bigl[n+(1+A)(\sqrt n\,\bar B+\bar B^4)\bigr],
 \label{eq:general-upper}\\
 U_{\mathrm{acc}}&\le
 C(1+A)(n+n^{3/4}B)\ell_{n,A,B}.
 \label{eq:accelerated-main}
\end{align}
The first algorithm satisfies
$\E\norm{\nabla\Phi_{1/(2L)}(\widehat x)}^2<\varepsilon^2/200$;
the second returns an $\varepsilon$-stationary point of
$\Phi_{1/(2L)}$ with probability at least $3/4$.
Both require only AS($L$) and average dual concavity.
\end{theorem}
The algorithms and complete proofs are in \cref{sec:upper,hd:sec:upper}.
The second guarantee is a probability guarantee with a checked abort
rule; we do not claim a squared-expectation guarantee for its abort
output. In the small-diameter regimes the first bound is used, so
there is no hidden logarithm. For $B\ge n^{1/4}$, the inequality
$n^{3/4}B\ge n$ absorbs the repeated full-pass cost in the second
bound.

\begin{theorem}[Complementary lower bounds under individual smoothness]
\label{thm:main}
For universal $A_0,c>0$, whenever $A\ge A_0$,
\begin{align}
 \T&\ge c\bigl[n+A\min\{\sqrt n\,\bar B,n^{3/4}\}\bigr],
 \label{eq:new-lower}\\
 \T&\ge c\bigl[n+A\min\{B^4,n^{3/4}B\}\bigr].
 \label{eq:replicated-main}
\end{align}
These bounds apply to the unrestricted randomized oracle model of
\cref{sec:model-main}, including exact function values.
\end{theorem}
The first bound combines the cone construction in \cref{app:cone}
with the smooth finite-sum baseline. The second follows from the
branch construction in \cref{app:branch} and the replication
reduction in \cref{repl:section}. The maximum of the two lower
bounds, combined with \cref{thm:upper-main} and class inclusion,
proves \cref{global:thm:main}; the complete comparison is in
\cref{sec:global}.

\paragraph{Accuracy range.}
The restriction $A\ge A_*$ is a universal nontrivial-accuracy
threshold, with no additional $n$- or $B$-dependent condition.
It is necessary to separate such a range from trivial accuracy:
$\norm{\nabla\Phi_{1/(2L)}(x^0)}^2\le4L\Delta$ implies that
$A\le1/4$ allows a zero-query output. We do not determine the sharp
constant transition between these two ranges.

\section{Related Work and Precise Comparisons}
\label{sec:related-main}
\paragraph{Nonconvex--concave upper bounds.}
Proximal methods attain the full-gradient polynomial rate
$AB=L^2D_Y\Delta\varepsilon^{-3}$ for Moreau stationarity:
Prox-DIAG~\cite[Theorem~2]{thekumparampil},
Lin et al.~\cite[Corollary~21]{lin2020}, and
Catalyst~\cite[Corollary~4.2]{catalyst}.
Zhao's Euclidean, unconstrained-primal, common-$L$ specialization
has the same high-accuracy scale~\cite[Corollary~4.2]{zhao2024}.
Their bounds retain logarithmic factors.
For finite sums, Yang et al.'s Corollary~4.3, with SVRG/SAGA,
gives the leading term $\widetilde O(A n^{3/4}B)$~\cite{catalyst}.
Our corresponding case is $B\ge n^{1/4}$, where
$n+A\rate(n,B)=\Theta(A n^{3/4}B)$.
Catalyst uses the arithmetic average of individual smoothness constants
and expected work; under $L_i\le L$, the polynomial scales agree.
Our upper bound instead uses mean-squared AS and a fixed query budget.

\begin{table}[p]
\caption{Prior benchmarks and special cases of our unified rate.
Here $A=L\Delta/\varepsilon^2$, $B=LD_Y/\varepsilon$,
$\bar B=\max\{1,B\}$, and our rows assume $A\ge A_*$.
The second rate line sets $L=\Delta=D_Y=1$, except in the smooth row,
where $L=\Delta=1$ and $D_Y=0$.}
\label{tab:prior-main}
\centering
\small
\setlength{\tabcolsep}{5pt}
\renewcommand{\arraystretch}{1.12}
\begin{tabularx}{\linewidth}{@{}>{\raggedright\arraybackslash}p{.25\linewidth}>{\raggedright\arraybackslash}p{.34\linewidth}>{\raggedright\arraybackslash}X@{}}
\toprule
\textbf{Work / regime} & \textbf{Rate} & \textbf{Correspondence / scope}\\
\midrule
\multicolumn{3}{@{}l}{\textbf{Prior benchmarks}}\\[4pt]
Full-gradient methods~\cite{thekumparampil,lin2020,catalyst,zhao2024}
& $\widetilde O(AB)$\newline
  $\widetilde O(\varepsilon^{-3})$
& NC--C upper bounds; common Euclidean setting.\\[5pt]
Yang et al.\ (2020)~\cite{catalyst}, Cor.~4.3
& $\widetilde O(A n^{3/4}B)$\newline
  $\widetilde O(n^{3/4}\varepsilon^{-3})$
& Finite-sum SVRG/SAGA upper bound; leading high-accuracy term.\\[5pt]
Wu et al.\ (2026 preprint)~\cite{qwu2026}
& $\Omega(AB)$\newline
  $\Omega(\varepsilon^{-3})$
& Deterministic, zero-respecting lower bound; constrained primal domain.\\[5pt]
Pan et al.\ (2026 preprint)~\cite{pzl}
& $\Theta(AB)$\newline
  $\Theta(\varepsilon^{-3})$
& Unrestricted deterministic full-gradient complexity; no logarithmic gap.\\[4pt]
\midrule
\multicolumn{3}{@{}l}{\textbf{This paper: unrestricted randomized component queries}}\\[4pt]
\multicolumn{3}{@{}p{\linewidth}@{}}{%
\textbf{Unified rate:}\quad
$\widetilde\Theta\!\left(n+A\min\{\sqrt n\,\bar B+\bar B^4,\,n^{3/4}\bar B\}\right)$
\quad(\cref{global:thm:main}).}\\[7pt]
\textbf{Single component}\newline
$n=1$, $B\ge1$
& $\widetilde\Theta(AB)$\newline
  $\widetilde\Theta(\varepsilon^{-3})$
& Matches the full-gradient upper scale and Pan's lower scale;
allows randomized queries.$^{\dagger}$\\[7pt]
\textbf{Small / zero diameter}\newline
$0\le B\le1$
& $\Theta(n+\sqrt n\,A)$\newline
  $\Theta(n+\sqrt n\,\varepsilon^{-2})$ at $D_Y=0$
& At $B=0$, recovers the smooth finite-sum benchmarks~\cite{spider,page,peng};
$n=1$ further gives $\Theta(A)$~\cite{carmon}.\\[7pt]
\textbf{First diameter regime}\newline
$1\le B\le n^{1/6}$
& $\Theta(n+A\sqrt n\,B)$\newline
  $\Theta(n+\sqrt n\,\varepsilon^{-3})$
& Resolves the first diameter-dependent branch; no logarithmic gap.\\[7pt]
\textbf{Intermediate regime}\newline
$n^{1/6}\le B\le n^{1/4}$
& $\Theta(n+AB^4)$\newline
  $\Theta(n+\varepsilon^{-6})$
& Identifies the $B^4$ crossover; no logarithmic gap.\\[7pt]
\textbf{Large diameter}\newline
$B\ge n^{1/4}$
& $\widetilde\Theta(n+A n^{3/4}B)$\newline
  $\widetilde\Theta(n+n^{3/4}\varepsilon^{-3})$
& Recovers Catalyst's polynomial upper scale and supplies a matching
randomized lower bound.\\[3pt]
\bottomrule
\end{tabularx}
\par\vspace{5pt}
\begin{minipage}{\linewidth}
\footnotesize
\textit{Conventions.} Prior NC--C rows show high-accuracy rates for
positive $D_Y$ under their original assumptions. Full-gradient and
component-query units coincide at $n=1$. Matching an upper rate does
not attribute a lower bound to the cited paper. Our upper bounds use
AS; our lower bounds already hold under IS. A tilde in our rows hides
only one upper-bound logarithm.
\par\vspace{2pt}
$^{\dagger}$Combining our $n=1$ lower bound with Pan et al.'s
$O((1+A)\bar B)$ upper bound~\cite[Theorem~3]{pzl} gives
\textbf{$\Theta(AB)$} for $B\ge1$. This log-free upper bound is inherited
from that work. The smooth lower bound is the imported
benchmark~\cite{peng}, with the transfer proved in \cref{sec:moreau-facts}.
\end{minipage}
\end{table}

\paragraph{Special cases of the unified rate.}
\Cref{tab:prior-main} places the established benchmarks and our
specializations in one table. The two endpoint cases are $n=1$,
$B\ge1$ for full-gradient NC--C optimization and $D_Y=0$ for smooth
finite sums. The large-diameter branch recovers Catalyst's polynomial
upper scale; the intervening branches describe the diameter crossover.
Under $L=\Delta=D_Y=1$, the first diameter range is
$n^{-1/6}\le\varepsilon\le1$ with
$\Theta(n+\sqrt n\,\varepsilon^{-3})=\Theta(n)$;
the intermediate range is
$n^{-1/4}\le\varepsilon\le n^{-1/6}$ with
$\Theta(n+\varepsilon^{-6})=\Theta(\varepsilon^{-6})$;
and the large-diameter range is $\varepsilon\le n^{-1/4}$.
All ranges are intersected with $\varepsilon^{-2}\ge A_*$.

\paragraph{Smooth finite sums and variance reduction.}
At $D_Y=0$, the rate $\Theta(n+\sqrt n A)$ matches the smooth
upper bounds of SPIDER~\cite[Theorem~2]{spider},
SpiderBoost~\cite[Theorem~1]{spiderboost},
SSRGD~\cite[Theorem~1]{ssrgd},
PAGE~\cite[Corollary~2]{page},
ProxSARAH with $\psi=0$~\cite[Theorem~6(c)]{proxsarah}, and
SARAH~\cite[author manuscript, Corollary~2]{sarah2022}.
For ProxSARAH, rounding the epoch count retains the additive $n$;
SARAH's squared-gradient tolerance is set to $\varepsilon^2$.
These are unconstrained first-order specializations.
For an $L$-smooth objective $F$, the direct and Moreau gradient norms
are equivalent within factors $1/2$ and $3/2$, as follows from the
argument in \cref{lem:transfer}.

The matching published lower benchmarks have specific scopes:
SPIDER's Theorem~3 assumes $2\le n\le cA^2$;
PAGE's Theorem~2 restricts algorithms to linear spans.
Zhou--Gu's Theorem~4.5 matches at weak curvature $\sigma=L$ and
$A\ge10^5\sqrt n$~\cite{zhougu2019};
Han et al.'s Theorem~68 matches at $\mu=L$, $n>270$, and
$A\ge435456\sqrt n$ in a restricted PIFO model~\cite{han2024}.
The latter is a numerical rate comparison, not a claim about
component-proximal access in our oracle.
At $n=1$, $D_Y=0$, the rate reduces to the classical
$\Theta(A)$ benchmark~\cite[Theorem~2, $p=1$]{carmon}.
Our unrestricted randomized IS baseline is imported from
Peng et al.~\cite[Theorem~3.1]{peng}; its Moreau and parameter-range
transfers are proved in \cref{sec:moreau-facts}.

\paragraph{Recent NC--C bounds.}
Wu, Gu, and Yang~\cite[Theorem~4.1]{qwu2026} prove an $AB$ lower bound
for deterministic zero-respecting algorithms on a constrained primal
domain. Pan, Zheng, and Li~\cite[Theorems~2--3]{pzl} obtain the
unrestricted deterministic one-component rate and a log-free upper bound.
Our lower bounds allow private randomization and arbitrary component
queries; Wu et al.'s full constrained-primal theorem is not a special
case of our $X=\R^d$ model.
Li et al.~\cite[Theorem~7(ii)]{foam2026} also attain the $AB$ polynomial
upper scale: at proximal initialization $z^0=x^0$, their regularized
gap satisfies $\Delta_{p_0}\le\Delta+O(\varepsilon^2/L)$.
These works and Peng et al.\ are preprints.
Our randomized lifting follows robust-chain and random-rotation ideas
of Carmon et al.~\cite[Sections~3.4 and~5]{carmon}, with componentwise
transcript accounting and a terminal-output test.
Related convex--concave acceleration results~\cite{luounbalanced}
impose componentwise curvature, whereas our inner-solver proof uses
strong monotonicity only of the average.

\paragraph{Nonconvex--strongly-concave problems.}
For NC--SC problems, the dual curvature parameter $\mu>0$ defines
$\kappa=L/\mu$, and the target is usually the direct primal gradient.
Zhang et al. study the condition-number dependence under
linear-span restrictions~\cite{zhang2021}. Zhou obtains
\begin{equation}\label{eq:zhou-main}
 \widetilde\Theta\!\left(n+A\min\{\sqrt n\,\kappa,
                                      n^{3/4}\sqrt\kappa\}\right)
\end{equation}
for $n\ge2$, $\kappa\ge1$, and $\varepsilon^2\le cL\Delta$,
for expected query counts in an AS, randomized zero-respecting
model with unconstrained dual variables and exact dual
initialization~\cite[Theorem~1]{zhou-as}. A fixed-budget
characterization is stated there for $\kappa\le\sqrt n$.
\Cref{tab:models-main} records the distinctions that prevent an
automatic transfer to our result.

\begin{table}[t]
\caption{Model distinctions from the NC--SC comparison in
Zhou~\cite{zhou-as}. The two settings share the AS convention and differ in the
features listed below.}
\label{tab:models-main}
\centering
\small
\setlength{\tabcolsep}{5pt}
\renewcommand{\arraystretch}{1.18}
\begin{tabularx}{\linewidth}{@{}>{\raggedright\arraybackslash}p{.24\linewidth}>{\raggedright\arraybackslash}X>{\raggedright\arraybackslash}X@{}}
\toprule
\textbf{Feature} & \textbf{This paper} & \textbf{Zhou (2026)} \\
\midrule
Dual structure & Concave average; bounded known $Y$ &
Strongly concave average; $Y=\R^{d_y}$ \\[7pt]
Initialization & Primal gap and dual diameter &
Primal gap and exact $y^*(x^0)$ \\[7pt]
Stationarity & $\norm{\nabla\Phi_{1/(2L)}}\le\varepsilon$ with probability $2/3$ &
$\E\norm{\nabla\Phi}^2\le\varepsilon^2$ \\[7pt]
Lower-bound algorithms & Unrestricted randomized & Randomized zero-respecting \\[7pt]
Oracle information & Component value and joint gradient & Component joint gradient \\[7pt]
Budget in main characterization & Fixed & Expected; fixed in the smaller-$\kappa$ regime \\
\bottomrule
\end{tabularx}
\end{table}

Dual regularization in NC--C creates an effective
$\kappa=\Theta(\bar B^2)$, but substituting this into
\eqref{eq:zhou-main} is only a formal rate comparison. For a smooth
value function with Lipschitz-gradient constant $L_\Phi$,
\begin{equation}\label{eq:stationarity-comparison-main}
 \norm{\nabla\Phi(x)}\le
 \left(1+\frac{L_\Phi}{2L}\right)
 \norm{\nabla\Phi_{1/(2L)}(x)}.
\end{equation}
Indeed, apply smoothness between $x$ and its proximal point, where
the two gradients agree. The factor can be of order $\kappa$ in
NC--SC. Our result therefore does not resolve the individually
smooth NC--SC gap identified by Zhou.

\paragraph{Different stochastic or stationarity models.}
Song and Xu study shared-sample stochastic queries with a
zero-respecting restriction~\cite{songxu}; Wendao Wu et al. study
unrestricted randomized NC--SC algorithms under fresh-gradient
sampling~\cite{wwu2026}. Neither oracle supplies persistent,
arbitrarily selectable component identities. Recent single-loop
finite-sum methods also target game-stationarity residuals rather
than the Moreau gradient of the maximized value~\cite{jiang2025}.
Their rates are not directly comparable without an oracle or
stationarity reduction. The 2026 results cited above are arXiv
preprints; our literature comparison uses the versions identified
in the bibliography.

\section{Algorithms and Upper-Bound Mechanisms}
\label{sec:algorithms-main}
Both algorithms use only two consequences of AS. First, the average
signed gradient $G=(\nabla_xf,-\nabla_yf)$ is $L$-Lipschitz.
Second, a fresh uniform minibatch of differences obeys
\begin{equation}\label{eq:variance-main}
 \E\left[\norm{\frac1b\sum_{j=1}^b
   [G_{I_j}(z)-G_{I_j}(z')]-(G(z)-G(z'))}^2
     \middle|\mathcal H\right]
 \le\frac{L^2}{b}\norm{z-z'}^2,
\end{equation}
where the points are measurable before the batch is sampled.
\Cref{as:lem:moments} proves this statement, its weighted version,
and stability under common regularization. No step size uses
$\sqrt nL$ as a replacement smoothness constant.

\subsection{Persistent estimates and independent primal averaging}
Choose $y^0\in Y$, and regularize the dual objective by
$-(\mu/2)\norm{y-y^0}^2$, with
$\mu\asymp\min\{L,\varepsilon^2/(LD_Y^2)\}$ and $\mu=L$ when
$D_Y=0$. Let $\Psi$ be its maximized value and
$\kappa=L/\mu=\Theta(\bar B^2)$. Since
$0\le\Phi-\Psi\le\mu D_Y^2/2$, Moreau stability controls the
resulting squared gradient error by $2L\mu D_Y^2$.

At a center $c_t$, we approximately solve
\[
 \min_x\max_{y\in Y}
 \left[f(x,y)-\frac\mu2\norm{y-y^0}^2+L\norm{x-c_t}^2\right].
\]
An estimate of $G(c_t,y^0)$ is retained across centers, refreshed
periodically, and updated by independent difference batches.
The inner loop uses fresh samples and averages its primal iterates.
\Cref{alg:averaging-main} summarizes the implementation; all
numerical constants and stopping horizons are in \cref{sec:upper}.

\begin{algorithm}[t]
\caption{Proximal averaging with a persistent estimator}
\label{alg:averaging-main}
\small
Choose the regularization $\mu$, $\kappa=L/\mu$, and a valid weighted
variance bound $\gamma\in[1,\kappa]$; use $\gamma=\kappa$ under AS.
Set $q=\max\{1,\lfloor\sqrt{n/\gamma}\rfloor\}$,
$b=\lceil2^{12}\gamma q\rceil$, and
$S=\lceil2^{12}(1+A)\rceil$. Initialize $c_0=x^0$.
\begin{enumerate}
\item At the first center of each epoch of $q$ centers, compute
$v_t=G(c_t,y^0)$ by a full pass. Otherwise update
\[
 v_t=v_{t-1}+\frac1b\sum_{i\in I_t}
 [G_i(c_t,y^0)-G_i(c_{t-1},y^0)].
\]
The batch is fresh and sampled with replacement; when $b>n$ use
exact differences. When $q=1$ every anchor is exact.
\item Starting from $z_0=(c_t,y^0)$, perform the projected updates
\eqref{eq:inner} for $K=\lceil2^{32}\kappa^2\rceil$ steps with
fresh independent indices and step size $\mu/(2^{16}L^2)$.
Let $c_{t+1}$ be the average of the $K$ primal iterates.
\item After $S$ centers, return $c_J$ for an independent uniform
$J\in\{0,\ldots,S-1\}$.
\end{enumerate}
\end{algorithm}

The existing anchor error $\xi_t=v_t-G(c_t,y^0)$ can be biased
conditional on $c_t$. The key is to analyze the corresponding
\emph{biased} saddle point. If $p_t$ and $p_{\xi_t}$ are the exact
primal solutions without and with that bias, then
\begin{equation}\label{eq:sensitivity-main}
 \norm{p_{\xi_t}-p_t}^2\le
 \frac{\norm{\xi_t^x}^2+\kappa\norm{\xi_t^y}^2}{L^2}.
\end{equation}
Fresh anchor increments form martingale differences, so their total
weighted error is bounded by
$\gamma q/b$ times the total squared movement of the centers.
The inner average converts this into a proximal error small enough
to absorb in the envelope descent inequality. The resulting cost is
\begin{equation}\label{eq:weighted-main}
 O\!\left(n+(1+A)(\sqrt{n\gamma}+\kappa^2)\right).
\end{equation}
AS gives $\gamma=\kappa$, proving \eqref{eq:general-upper}.
For the common-interaction subclass
$f_i(x,y)=H_0(x,y)+g_i(x)+r_i(y)$, the centered dual increments
vanish and $\gamma=1$. This yields the additional bound
$O(n+(1+A)(\sqrt n+\bar B^4))$; no subclass matching claim is made.

\subsection{Accelerated dual proximal steps with a primal certificate}
For a primal center $c$, define
$F_c(x,y)=f(x,y)+L\norm{x-c}^2$ and the convex dual objective
$H_c(y)=-\min_xF_c(x,y)$. Acceleration is applied to this convex
objective through dual proximal subproblems with parameter
$\tau=L/\sqrt n$. It is not applied to a general monotone operator
as though that operator were a gradient.

A projected-SVRG solver handles each strongly convex--strongly
concave subproblem. Under mean-squared operator smoothness $\ell$
and average strong monotonicity $\tau$, its epoch length is
$O(\ell^2/\tau^2)$. Here $\ell=4L$ and hence each epoch costs
$O(n)$. Two full passes construct and check an actual normal-cone
residual. The primal domain used internally is a public compact
ball derived from the initial primal gap and $D_Y$, rather than an
extra assumption on the original problem.

The main certificate, proved in \cref{hd:lem:accelerated-primal}, is
\begin{equation}\label{eq:primal-certificate-main}
 \max_{y\in Y}F_c(\bar x,y)-\min_{x\in X_c}F_c(x,y_K)
 \le\frac{2\tau D_Y^2}{(K+1)^2}
       +K(\diam(X_c)+D_Y)\zeta.
\end{equation}
Here $\bar x$ is an explicitly weighted primal average and $\zeta$
is the checked residual tolerance for each inner solve. This is a
primal--dual gap, not merely a dual objective gap. Taking
$K=O(1+B/n^{1/4})$ makes the error $O(\varepsilon^2/L)$.
There are $O(1+A)$ outer centers, and a public residual-failure
budget produces a single logarithmic factor. This proves
\eqref{eq:accelerated-main}. All oracle evaluations remain in $Y$
even when an extrapolated dual center lies outside it.

\section{Lower-Bound Mechanisms}
\label{sec:lower-main}
The two lower families are independent constructions. They use
different local normalizations but are scaled to the same physical
$n,L,\Delta,D_Y,\varepsilon$ and the same Moreau parameter.

\subsection{A cone construction for the small-diameter regime}
In each of $M$ stages, sample a Haar orthonormal frame
$w_{1j},\ldots,w_{nj}$ and define the unit source
$v_j=n^{-1/2}\sum_iw_{ij}$. A component query can disclose one row
of the current frame. The dual coupling is affine and arranged so
that maximizing over a single Euclidean ball produces the norm
\begin{equation}\label{eq:cone-main}
 \left(\sum_{j=1}^M G_{j-1}^2
             \norm{\zeta(z_j)-v_j}^2\right)^{1/2},
\end{equation}
where $\zeta$ is a smooth bounded radial map and $G_j$ is a smooth
exit gate. A separate clock potential contains the source latch
additively inside its scalar activation. This placement keeps
individual smoothness bounded while the clock has scale
$q=\sqrt n/R$, with dual radius $R$.

The stationarity certificate has two steps. Every sufficiently small
subgradient first forces all gates to be exactly zero or one. If
the residual in \eqref{eq:cone-main} is nonzero, its normalized
gradient has a nonnegative inner product with every correction and
therefore cannot be small. The residual must vanish, which then
forces every source to be recovered and every clock to finish.
Moreau stationarity implies approximate recovery at the original
output through the distance to its proximal point.

For unrestricted queries, an ideal transcript censors the hidden
parts of unfinished stages. A uniform conditional spherical estimate
shows that their scalar values and complete first derivatives are
exactly inactive until a constant fraction of the frame rows has
been disclosed. Applying the estimate to normalized query vectors
also covers arbitrary query norms and the unqueried output. Thus a
stage costs $\Omega(n)$ calls. The gap is $O(Mq)$, so after scaling
$M\asymp A/q$ gives $\Omega(A\sqrt nB)$ until the smoothness cap
$R\lesssim n^{1/4}$. If the permitted gap is too small for even one
stage at this scale, the desired term is $O(n)$ and a separate
hidden-component instance supplies it. The formulas and full proof
are in \cref{app:cone}.

\subsection{Serial orthogonal branches for the large-diameter regime}
A branch edge has $n$ individually bounded linear maps with mutually
orthogonal images. Their average is an isometry scaled by
$n^{-1/2}$. Random rotations hide the vertex coordinate systems
before the algorithm starts. A stage containing $N-1$ such edges is
embedded into a quadratic dual module. Its eliminated endpoint
energy is uniformly coercive, while its dual optimizer has norm
of order $n^{1/4}N$ times the connector norm.

The information lower bound allows a stronger oracle than a
component query: in one round it can act on every edge, in both
directions, with separately chosen indices. All inputs in a round
are chosen before any reply from that round. It also releases a source
vector once a specified fraction of its coordinates is reachable.
If $p_t(K)$ is the reachable mass at depth $t$ after $K$ rounds and
$r_K$ releases, then
\begin{equation}\label{eq:mass-main}
 p_t(K)\le\min\left\{1,\rho^{-r_K}\frac{\binom Kt}{n^t}\right\}.
\end{equation}
The binomial term enforces chronological propagation. A release
may increase existing reachable mass by at most $\rho^{-1}$;
cached probes from before the release remain in the accounting.
When there are $M$ stages and $H=M(N-1)$ edges, the release cost is
dominated by the propagation deficit for $K<c nH$.

A lazy-rotation argument represents all new arbitrary inputs using
independent Gaussian directions. A sufficiently large finite
multiplicity makes their accidental terminal correlation small.
This yields $\Omega(nH)$ calls even with arbitrary transpose probes,
source releases, and an unqueried terminal output. The underlying
instance law is a fixed product-Haar distribution, independent of
the algorithm's private seed.

The nonlinear outer assembly has a single possible frontier at a
small-gradient point. This fact bounds the \emph{total} connector
norm independently of $M$. Connector coercivity is proved for the
constrained dual value first; only afterwards is the unconstrained
dual optimizer shown to lie in the prescribed ball. Thus the domain
radius is $O(n^{1/4}N)$ rather than acquiring a $\sqrt M$ factor.
The initial gap is $O(M)$, giving
\begin{equation}\label{eq:scaling-main}
 M\asymp A,\quad N\asymp B/n^{1/4},\quad
 nM(N-1)\asymp A n^{3/4}B
\end{equation}
in the high-diameter range. Exact flatness of every hidden nonlinear
value and first derivative connects the matrix experiment to the
actual component oracle. \Cref{app:branch} provides the complete
construction and simulation.

\subsection{Why the intermediate coefficient is \texorpdfstring{$B^4$}{B to the fourth power}}
For $B$ above an absolute constant, use an $s$-component hard instance with
$s\asymp\min\{n,B^4\}$ so that $B$ is large relative to $s^{1/4}$.
Balanced replication creates exactly $n$ components while preserving
the average objective and requiring one original query per
simulated component query. The high-diameter lower bound becomes
\[
 A s^{3/4}B\asymp A\min\{B^4,n^{3/4}B\}.
\]
When $s$ does not divide $n$, the replication weights differ. We
reserve a factor of two in component smoothness and explicitly
compare the resulting two Moreau parameters. This proves the
reduction with integer component counts and the original success
criterion; see \cref{repl:balanced-replication,repl:lower-envelope}.

\section{Scope and Further Questions}
\label{sec:limitations-main}
The result concerns exact component-query complexity. It does not
bound arithmetic time, bit complexity, storage, or the cost of
projecting onto an arbitrary known convex domain. The lower-bound
dimension is finite but may grow exponentially with the serialized
branch length. Obtaining comparable bounds in a fixed or polynomially
bounded dimension remains a separate question.

The equality of IS and AS rates uses average dual concavity. Imposing
dual concavity on every component gives a smaller class, and the
branch lower construction does not establish the same characterization
there. The common-interaction refinement in \eqref{eq:weighted-main}
also indicates that additional structure can change the diameter
dependence. Removing the remaining logarithm for the general
large-diameter finite-sum class, and determining sharp constants in
the transition from zero-query to nontrivial accuracy, are further
questions. No empirical performance or practical tuning claim is
needed for the oracle characterization.

\clearpage
\appendix
\part*{Appendix: Complete Proofs}
The appendices supply all construction and algorithmic proofs used
in the main theorem. The only external lower-bound input is the
smooth finite-sum theorem of Peng et al.~\cite{peng}; its precise
oracle and stationarity transfer is proved in
\cref{sec:moreau-facts}. Numerical constants in the two lower
constructions are local to their respective normalizations.

\section{Averaged Smoothness and Moreau Calculus}
\label{app:calculus}
Throughout the algorithmic proofs, the normal cone uses the outward convention
\[
 N_C(z)=\{v:\langle v,u-z\rangle\le0\text{ for every }u\in C\}.
\]
For a positive semidefinite matrix $W$, write $\|v\|_W^2=v^\mathsf TWv$.
All expectations conditional on a history are taken before the fresh
indices in question are sampled. In particular, a current estimator
error is not assumed centered conditional on an iterate that depends on it.

\subsection{Mean-squared bounds and regularized operators}
\begin{lemma}[Consequences of averaged smoothness]\label{as:lem:moments}
Suppose the signed component gradients $G_i=(\nabla_x f_i,-\nabla_y f_i)$
satisfy
\begin{equation}\label{as:eq:moment}
 \frac1n\sum_{i=1}^n\norm{G_i(z)-G_i(z')}^2
 \le L^2\norm{z-z'}^2
 \qquad(z,z'\in\mathbb R^{d_x}\times Y).
\end{equation}
Then the following statements hold with the same $L$.
\begin{enumerate}[label=(\roman*),leftmargin=2em]
\item The average $G=n^{-1}\sum_iG_i$ is $L$-Lipschitz. The average
function $f$ has a jointly $L$-Lipschitz gradient, and
$\Phi(x)=\max_{y\in Y}f(x,y)$ is $L$-weakly convex.
\item Let $z,z'$ be measurable with respect to a history $\mathcal H$.
For independent uniform indices $I_1,\ldots,I_b$, fresh relative to
$\mathcal H$, write
\[
 d_i=G_i(z)-G_i(z'),\quad \bar d=\frac1n\sum_i d_i,\quad
 \xi=\frac1b\sum_{j=1}^b d_{I_j}-\bar d.
\]
Then $\mathbb E[\xi\mid\mathcal H]=0$ and
\begin{equation}\label{as:eq:conditional}
 \mathbb E[\norm{\xi}^2\mid\mathcal H]
 =\frac1b\left(\frac1n\sum_i\norm{d_i}^2-\norm{\bar d}^2\right)
 \le\frac{L^2}{b}\norm{z-z'}^2.
\end{equation}
For $W_\kappa=\operatorname{diag}(I_{d_x},\kappa I_{d_y})$,
$\kappa\ge1$, the corresponding weighted estimate is
\begin{equation}\label{as:eq:weighted}
 \frac1n\sum_i\norm{d_i-\bar d}_{W_\kappa}^2
 \le \kappa L^2\norm{z-z'}^2.
\end{equation}
\item Let $R$ be a common $r$-Lipschitz map, and $b_0$ any common
constant vector. The operators $T_i(z)=G_i(z)+R(z)+b_0$ satisfy
\begin{equation}\label{as:eq:regularized}
 \left(\frac1n\sum_i\norm{T_i(z)-T_i(z')}^2\right)^{1/2}
 \le (L+r)\norm{z-z'}.
\end{equation}
No bound on the size of $b_0$ or the affine component offsets is required.
\end{enumerate}
\end{lemma}
\begin{proof}
The sign change on the dual gradient is an orthogonal transformation.
Jensen's inequality gives
\[
 \norm{G(z)-G(z')}^2
 \le \frac1n\sum_i\norm{G_i(z)-G_i(z')}^2.
\]
For each $y$, the function $f(\cdot,y)+(L/2)\norm{\cdot}^2$ is convex;
taking its maximum over $Y$ proves weak convexity of $\Phi$.
Conditional on $\mathcal H$, the points are fixed and the samples are
i.i.d.; expanding the squared norm proves \eqref{as:eq:conditional}.
This argument is conditional on the history \emph{before} the fresh
indices are drawn, not on a later iterate depending on them.
For any positive semidefinite $W$, centering gives
\[
 \frac1n\sum_i\norm{d_i-\bar d}_W^2
 =\frac1n\sum_i\norm{d_i}_W^2-\norm{\bar d}_W^2.
\]
Use $W_\kappa\preceq\kappa I$ and \eqref{as:eq:moment} to obtain
\eqref{as:eq:weighted}. Finally, Minkowski's inequality in the Hilbert
space with norm $(n^{-1}\sum_i\norm{a_i}^2)^{1/2}$ yields
\[
 \left(\frac1n\sum_i\norm{(G_i(z)-G_i(z'))+(R(z)-R(z'))}^2\right)^{1/2}
 \le L\norm{z-z'}+\norm{R(z)-R(z')}.
\]
The common constant cancels, proving \eqref{as:eq:regularized}.
\end{proof}

\subsection{Weak convexity, perturbation, and stationarity transfer}
\begin{lemma}[Weak convexity and perturbation]\label{lem:moreau}\label{as:lem:moreau}
The function \(\Phi\) is \(L\)-weakly convex. Its unique proximal point
\(p(c)=\argmin_x\{\Phi(x)+L\norm{x-c}^2\}\) satisfies
\[
\nabla\Phi_{1/(2L)}(c)=2L(c-p(c)),\qquad
\norm{\nabla\Phi_{1/(2L)}(x^0)}^2\le4L\Delta.
\]
In particular, \(A\le1/4\) permits a zero-query output.
If \(\phi,\psi\) are bounded-below \(L\)-weakly convex functions and
\(0\le\phi-\psi\le r\), then
\[
\norm{\nabla\phi_{1/(2L)}(c)-\nabla\psi_{1/(2L)}(c)}^2\le4Lr.
\]
\end{lemma}
\begin{proof}
Each \(f(\cdot,y)+(L/2)\norm{\cdot}^2\) is convex; taking the maximum
proves weak convexity. The proximal objective is \(L\)-strongly convex
and coercive. The gradient identity follows by differentiating its
uniquely minimized value. Minimality gives
\(L\norm{p(x^0)-x^0}^2\le\Phi(x^0)-\Phi(p(x^0))\le\Delta\).
For the perturbation statement, let \(p,q\) be the two proximal points.
Add the two strong-convexity inequalities evaluated at the opposite
minimizer:
\[
L\norm{p-q}^2\le(\phi-\psi)(q)-(\phi-\psi)(p)\le r.
\]
Multiply by \(4L\).
\end{proof}

\begin{lemma}[Smooth-to-Moreau transfer]\label{lem:transfer}
For a bounded-below \(L\)-smooth function \(F\),
\[
\norm{\nabla F(x)}\le\frac32\norm{\nabla F_{1/(2L)}(x)}.
\]
\end{lemma}
\begin{proof}
At its proximal point \(p\),
\[
\nabla F(p)=2L(x-p)=\nabla F_{1/(2L)}(x).
\]
Smoothness gives
\[
\norm{\nabla F(x)-\nabla F(p)}\le L\norm{x-p}
=\frac12\norm{\nabla F(p)}.
\]
\end{proof}

\begin{lemma}[Comparison of two Moreau gradients]
\label{repl:moreau-comparison}
Let $F:\mathbb R^d\to\mathbb R$ be continuous, bounded below, and
$\rho$-weakly convex, with $\rho\ge0$.  If
$0<\lambda\le\nu$ and $\rho\nu<1$, then
\begin{equation}
 \|\nabla F_\nu(x)\|
 \le \frac{1-\rho\lambda}{1-\rho\nu}
          \|\nabla F_\lambda(x)\|
 \qquad(x\in\mathbb R^d).
 \label{repl:moreau-inequality}
\end{equation}
In particular, for $\rho=L/2$, $\lambda=1/(2L)$, and $\nu=1/L$,
\begin{equation}
 \|\nabla F_{1/L}(x)\|
 \le \tfrac32\|\nabla F_{1/(2L)}(x)\|.
 \label{repl:factor-three-halves}
\end{equation}
\end{lemma}
\begin{proof}
For $t\rho<1$, the proximal objective is strongly convex and has a unique
minimizer, denoted $P_t(x)$.  Hypomonotonicity of the subdifferential of
$F$ gives
\[
 \|P_t(u)-P_t(v)\|\le(1-\rho t)^{-1}\|u-v\|.
\]
Indeed, the proximal optimality conditions and
$\langle a-b,p-q\rangle\ge-\rho\|p-q\|^2$ for
$a\in\partial F(p)$, $b\in\partial F(q)$ imply
\[
 (1-\rho t)\|P_t(u)-P_t(v)\|^2
 \le\langle u-v,P_t(u)-P_t(v)\rangle.
\]
Write $p=P_\lambda(x)$ and
$g=\nabla F_\lambda(x)=(x-p)/\lambda\in\partial F(p)$.
The same optimality condition at parameter $\nu$ implies
$p=P_\nu(p+\nu g)$.  Since $x=p+\lambda g$, proximal Lipschitzness gives
\[
 \|P_\nu(x)-p\|
 \le\frac{\nu-\lambda}{1-\rho\nu}\|g\|.
\]
Consequently,
\[
 \|\nabla F_\nu(x)\|
 =\frac1\nu\|x-P_\nu(x)\|
 \le\left(\frac\lambda\nu+
       \frac{\nu-\lambda}{\nu(1-\rho\nu)}\right)\|g\|
 =\frac{1-\rho\lambda}{1-\rho\nu}\|g\|.
\]
The three-halves factor follows by direct substitution.
\end{proof}

\section{Proximal Averaging with a Persistent Estimator}\label{sec:upper}
Throughout this section assume AS as in \eqref{as:eq:moment} and
the model of \cref{sec:model-main}. We give all algorithmic parameters,
query budgets, and probability guarantees explicitly. The same argument
also gives a structural refinement.
For the common-interaction subclass
\begin{equation}\label{eq:common-form}
f_i(x,y)=H_0(x,y)+g_i(x)+r_i(y),
\end{equation}
it yields
\begin{equation}\label{eq:common-upper}
T\le C[n+(1+A)(\sqrt n+\barB^4)].
\end{equation}
The cone hard family is outside this subclass.

\subsection{Regularization and weighted centered variance}
Choose \(y^0\in Y\). For \(D_Y>0\), set
\[
\mu=\min\{L,\eps^2/(2^{20}LD_Y^2)\},\qquad
\kappa=L/\mu=\max\{1,2^{20}B^2\};
\]
for \(D_Y=0\), set \(\mu=L\). Let
\[
f_\mu(x,y)=f(x,y)-\frac\mu2\norm{y-y^0}^2,\quad
\Psi(x)=\max_Y f_\mu(x,y),\quad \mathcal E=\Psi_{1/(2L)}.
\]
Then, writing \(r=\mu D_Y^2/2\),
\begin{equation}\label{eq:mu-transfer}
0\le\Phi-\Psi\le r\le\frac{\eps^2}{2^{21}L},\qquad
\norm{\nabla\Phi_{1/(2L)}-\nabla\mathcal E}^2\le\frac{\eps^2}{2^{19}}.
\end{equation}
The second inequality follows from \cref{as:lem:moreau}; both functions
are $L$-weakly convex by AS of the components and Jensen's inequality.
Define signed gradients
\(G_i=(\nabla_xf_i,-\nabla_yf_i)\) and \(G=n^{-1}\sum_iG_i\).
The sign change is an isometry, so the $G_i$ satisfy the same
mean-squared AS inequality as the original gradients. The average $G$
is $L$-Lipschitz by \cref{as:lem:moments}.
For anchors \((c,y^0),(c',y^0)\), put
\(d_i=G_i(c',y^0)-G_i(c,y^0)\), \(\bar d=n^{-1}\sum_i d_i\).
Assume a known \(\gamma\in[1,\kappa]\) satisfies
\begin{equation}\label{eq:gamma}
\frac1n\sum_i\left(\norm{d_i^x-\bar d^x}^2+
\kappa\norm{d_i^y-\bar d^y}^2\right)
\le\gamma L^2\norm{c'-c}^2.
\end{equation}
Universally $\gamma=\kappa$ works already under AS: centering in
the weighted norm and using $W_\kappa\preceq\kappa I$ gives
\[
 \frac1n\sum_i\norm{d_i-\bar d}_{W_\kappa}^2
 \le\kappa\frac1n\sum_i\norm{d_i}^2
 \le\kappa L^2\norm{c'-c}^2.
\]
Under \cref{eq:common-form}, the dual increments are identical across
components, so the weighted variance is just the centered primal
variance; AS then gives $\gamma=1$. No explicit decomposition is needed
by the algorithm.

\subsection{A biased proximal saddle}
For center \(c\), define
\[
Q_c(x,y)=f_\mu(x,y)+L\norm{x-c}^2,\qquad
T_c(x,y)=G(x,y)+(2L(x-c),\mu(y-y^0)).
\]
The operator is \(3L\)-Lipschitz and satisfies
\begin{equation}\label{eq:monotone}
\ip{T_c(z)-T_c(z')}{z-z'}
\ge L\norm{x-x'}^2+\mu\norm{y-y'}^2.
\end{equation}
Indeed, the strong convexity in \(x\) and strong concavity in \(y\) give
\[
\ip{T_c(z)}{z-z'}\ge
Q_c(x,y')-Q_c(x',y)+\frac L2\norm{x-x'}^2+
\frac\mu2\norm{y-y'}^2.
\]
Interchanging \(z,z'\) and adding proves \cref{eq:monotone}.
For every fixed vector \(\xi=(\xi_x,\xi_y)\), a sufficiently small-step
projection map for \(T_c+\xi\) is a contraction on
\(\R^{d_x}\times Y\), so it has a unique fixed point. This is the saddle
solution for the correspondingly tilted objective.
Write the unbiased and biased solutions as
\(z_*=(p,y_*)\) and \(z_\xi=(p_\xi,y_\xi)\).

\begin{lemma}[Primal sensitivity]\label{lem:bias}
\[
\norm{p_\xi-p}^2\le
\frac{\norm{\xi_x}^2}{L^2}+\frac{\norm{\xi_y}^2}{L\mu}
=:\mathsf B(\xi).
\]
\end{lemma}
\begin{proof}
Test each variational inequality at the other solution and add.
With \(W=L\norm{p_\xi-p}^2+\mu\norm{y_\xi-y_*}^2\),
\[
W\le-\ip{\xi}{z_\xi-z_*}
\le\frac W2+
\frac12\left(\frac{\norm{\xi_x}^2}{L}+
\frac{\norm{\xi_y}^2}{\mu}\right).
\]
Divide by \(L\) and discard the dual term.
\end{proof}

\subsection{Independent inner averaging}
Given an anchor estimate \(v=G(c,y^0)+\xi\), initialize
\(z_0=(c,y^0)\). With fresh independent uniform indices \(I_j\), use
\begin{align}
V_j&=v+G_{I_j}(z_j)-G_{I_j}(c,y^0)
+(2L(x_j-c),\mu(y_j-y^0)),\label{eq:inner}\\
z_{j+1}&=\Pi_{\R^{d_x}\times Y}(z_j-\eta V_j),\qquad
c^+=K^{-1}\sum_{j=0}^{K-1}x_j,\nonumber\\
\eta&=\frac{\mu}{2^{16}L^2},\qquad K=\lceil2^{32}\kappa^2\rceil.
\nonumber
\end{align}
Each iteration uses at most two component queries.

\begin{lemma}[Inner error]\label{lem:inner}
Conditional on the entire anchor history,
\begin{align}
\E\norm{c^+-p_\xi}^2
&\le\frac3{2^{16}\kappa}
(\norm{p_\xi-c}^2+D_Y^2),\label{eq:inner-biased}\\
\E\norm{c^+-p}^2
&\le\frac1{1024}\norm{p-c}^2+
3\mathsf B(\xi)+\frac{\eps^2}{2^{32}L^2}.
\label{eq:inner-error}
\end{align}
\end{lemma}
\begin{proof}
Write \(V_j=T_c(z_j)+\xi+\zeta_j\). Conditional on \(z_j\),
\(\E_j\zeta_j=0\) and
\(\E_j\norm{\zeta_j}^2\le L^2\norm{z_j-(c,y^0)}^2\).
This is \cref{as:lem:moments} with one fresh uniform index, conditional
on the entire anchor history and the inner iterates before $I_j$ is
drawn. The possibly noncentered anchor bias $\xi$ is held fixed; no
conditioning on a future sampled iterate is used.
Subtract the projection fixed-point identity
\[
z_\xi=\Pi_{\R^{d_x}\times Y}
\bigl(z_\xi-\eta(T_c(z_\xi)+\xi)\bigr).
\]
In particular, no vanishing constrained gradient is assumed.
For \(e_j=z_j-z_\xi\) and
\(V_*=\norm{p_\xi-c}^2+D_Y^2\), nonexpansiveness gives
\begin{align*}
\E_j\norm{e_{j+1}}^2
&\le\norm{e_j}^2
-2\eta(L\norm{x_j-p_\xi}^2+\mu\norm{y_j-y_\xi}^2)
+11\eta^2L^2\norm{e_j}^2+2\eta^2L^2V_*\\
&\le\norm{e_j}^2-\eta L\norm{x_j-p_\xi}^2
-\eta\mu\norm{y_j-y_\xi}^2+2\eta^2L^2V_*.
\end{align*}
The first inequality uses \(3L\)-Lipschitz continuity and
\(\norm{z_j-(c,y^0)}^2\le2\norm{e_j}^2+2V_*\); the second uses
\(11\eta L^2\le\mu\). Sum over \(j\), use \(\norm{e_0}^2\le V_*\),
and apply Jensen:
\[
\E\norm{c^+-p_\xi}^2\le
\left(\frac1{\eta LK}+2\eta L\right)V_*
\le\frac3{2^{16}\kappa}V_*.
\]
Apply the squared triangle inequality twice and \cref{lem:bias}.
The coefficient of \(\norm{p-c}^2\) is at most \(12/2^{16}\le1/1024\),
and the coefficient of \(\mathsf B(\xi)\) is at most \(3\).
The remaining term is
\[
\frac6{2^{16}\kappa}D_Y^2
\le\frac{6\eps^2}{2^{36}L^2}\le\frac{\eps^2}{2^{32}L^2}.
\]
No bound on component affine gradients is used: they cancel in the
differences.
\end{proof}

\subsection{Outer recursion}
Choose
\[
q=\max\{1,\lfloor\sqrt{n/\gamma}\rfloor\},\quad
b=\lceil2^{12}\gamma q\rceil,\quad
S=\lceil2^{12}(1+A)\rceil.
\]
At the beginning of every epoch of at most \(q\) centers, compute
\(v_t=G(c_t,y^0)\) by one full pass. At the other centers, draw a fresh
batch of \(b\) indices with replacement and set
\[
v_t=v_{t-1}+\frac1b\sum_{i\in I_t}
\bigl[G_i(c_t,y^0)-G_i(c_{t-1},y^0)\bigr].
\]
If \(b>n\), replace all difference batches in the epoch by exact
full-gradient differences, costing \(2n\le2b\).
Run the independent inner loop to obtain \(c_{t+1}\).
Return \(c_J\), where \(J\) is independent uniform on \(\{0,\ldots,S-1\}\).

Let \(\xi_t=v_t-G(c_t,y^0)\) and
\(\norm{\xi}_\kappa^2=\norm{\xi_x}^2+\kappa\norm{\xi_y}^2\).
Fresh estimator increments are conditionally centered, with conditional
weighted variance at most
\(\gamma L^2\norm{c_t-c_{t-1}}^2/b\).
Cross terms with the preceding estimator error have expectation zero.
Summing the martingale increments within each epoch yields
\begin{equation}\label{eq:outer-variance}
\sum_t\E\mathsf B(\xi_t)\le
\frac{\gamma q}{b}\sum_t\E\norm{c_{t+1}-c_t}^2.
\end{equation}
At refreshes the error is zero. Exact differences introduce no noise.
The full current error is not asserted to be conditionally centered
given the current center.

\subsection{Descent and query accounting}
Let \(p_t=\argmin_x\{\Psi(x)+L\norm{x-c_t}^2\}\) and set
\[
\mathsf D=\sum_t\E\norm{p_t-c_t}^2,\qquad
\mathsf E=\sum_t\E\norm{c_{t+1}-p_t}^2.
\]
The total squared movement is at most \(2\mathsf D+2\mathsf E\).
By \cref{eq:inner-error,eq:outer-variance},
\[
\mathsf E\le\frac{\mathsf D}{1024}
+\frac{3\gamma q}{b}(2\mathsf D+2\mathsf E)
+\frac{S\eps^2}{2^{32}L^2}
\le\frac{\mathsf D}{128}+\frac{S\eps^2}{2^{31}L^2}.
\]
The last inequality follows by moving the \(\mathsf E\) term left and
using \(b\ge2^{12}\gamma q\).
Pointwise,
\[
\mathcal E(c_{t+1})\le\mathcal E(c_t)
-L\norm{p_t-c_t}^2+L\norm{c_{t+1}-p_t}^2.
\]
Since \(\mathcal E(c_0)-\inf\Psi\le\Delta+r\), summing gives
\[
\frac{\mathsf D}{S}\le\frac{128}{127}
\left(\frac{\Delta+r}{LS}+\frac{\eps^2}{2^{31}L^2}\right).
\]
Therefore \(\E\norm{\nabla\mathcal E(c_J)}^2=4L^2\mathsf D/S
\le\eps^2/512\). By \cref{eq:mu-transfer},
\[
\E\norm{\nabla\Phi_{1/(2L)}(c_J)}^2
\le\frac{\eps^2}{256}+\frac{\eps^2}{2^{18}}<\frac{\eps^2}{200}.
\]
Markov gives success at least \(199/200\).

The query count is at most \(n(1+S/q)+2bS+2KS\).
For \(q\ge2\), \(n/q,b=O(\sqrt{n\gamma})\).
For \(q=1\), every anchor is exact and no difference batches are used;
then \(n=O(\sqrt{n\gamma})\). Consequently
\begin{equation}\label{eq:weighted-upper}
T\le C\left[n+(1+A)(\sqrt{n\gamma}+\kappa^2)\right].
\end{equation}
Taking \(\gamma=\kappa\) or \(\gamma=1\) proves
\cref{eq:general-upper,eq:common-upper}. All constants and loop lengths
are universal and prescribed; no logarithm is hidden.

\section{Accelerated Proximal Dual Steps under Averaged Smoothness}\label{hd:sec:upper}
Throughout this section we use only AS, average dual concavity, and
the bounded-domain and initial-gap assumptions of \cref{sec:model-main}.
Write $\log_+(u)=\max\{0,\log u\}$ for $u>0$. This section does not treat a general strongly monotone operator as a gradient. A simple projected finite-sum solver is used only on strongly convex--strongly concave saddle subproblems. Acceleration is applied to a genuine convex dual objective, and a primal averaging certificate is proved explicitly. In particular, a small dual objective gap alone is \emph{not} substituted for a small primal proximal gap.

\subsection{An elementary projected SVRG solver for a strongly monotone sum}
Let $\mathcal K$ be a nonempty compact convex product domain of diameter at most $d$. Let $T=n^{-1}\sum_iT_i$ be $\mu$-strongly monotone on $\mathcal K$ and suppose only
\begin{equation}\label{as:eq:svrg-ms}
 \frac1n\sum_i\norm{T_i(z)-T_i(z')}^2\le\ell^2\norm{z-z'}^2
 \qquad(z,z'\in\mathcal K).
\end{equation}
Let $z_*$ solve $0\in T(z_*)+N_{\mathcal K}(z_*)$. Normal cones use the outward convention. Individual $T_i$ need not be monotone or $\ell$-Lipschitz. Jensen's inequality gives the required $\ell$-Lipschitz continuity of the \emph{average} $T$.
At a snapshot $a\in\mathcal K$, query and cache all $T_i(a)$. Starting at $z_0=a$, use fresh independent uniform indices and
\begin{equation}\label{hd:eq:svrg-operator}
 v_t=T_{I_t}(z_t)-T_{I_t}(a)+T(a),\qquad
 z_{t+1}=\Pi_{\mathcal K}(z_t-hv_t),\qquad h=\frac\mu{12\ell^2}.
\end{equation}
One epoch has $m=\lceil12\ell^2\log(6)/\mu^2\rceil$ steps and costs at most $n+m$ calls.

\begin{lemma}[Contraction and a computable normal-cone residual]\label{hd:lem:svrg-residual}
One epoch satisfies
\[
 \E[\norm{z_m-z_*}^2\mid a]\le\frac13\norm{a-z_*}^2.
\]
After $J$ epochs, two additional full passes produce $z^+\in\mathcal K$ and a computable vector $r\in T(z^+)+N_{\mathcal K}(z^+)$ such that
\[
 \PP(\norm r>\zeta)\le \frac{36\ell^2d^2}{\zeta^2}\,3^{-J}.
\]
Thus a fixed budget $O((n+\ell^2/\mu^2)[1+\log_+(\ell d/(\zeta\sqrt p))])$ makes the residual at most $\zeta$ with probability at least $1-p$, where $\log_+u=\max\{0,\log u\}$ and a positive diameter upper bound $d$ is used. The returned residual can be checked exactly; no unknown optimum is used in the check.
\end{lemma}
\begin{proof}
The variational inequality is equivalent to $z_*=\Pi_{\mathcal K}(z_*-hT(z_*))$. Conditional centering of $v_t$ and nonexpansiveness of projection give, with $r_t=\norm{z_t-z_*}$ and $r_a=\norm{a-z_*}$,
\[
 \E_t r_{t+1}^2
 \le(1-2\mu h+3\ell^2h^2)r_t^2+2\ell^2h^2r_a^2
 \le(1-\mu h)r_t^2+2\ell^2h^2r_a^2.
\]
Indeed, conditional on the history immediately before drawing $I_t$, both $z_t$ and $a$ are fixed. The exact identity
\[
 \mathbb E_t\norm{v_t-T(z_t)}^2
 =\frac1n\sum_i\norm{T_i(z_t)-T_i(a)}^2-\norm{T(z_t)-T(a)}^2
 \le\ell^2\norm{z_t-a}^2\le2\ell^2(r_t^2+r_a^2)
\]
follows from \eqref{as:eq:svrg-ms}. Jensen's inequality also gives
$\norm{T(z_t)-T(z_*)}^2\le\ell^2r_t^2$.
These are the only smoothness estimates in the one-step recurrence;
there is no pathwise component-contraction assertion.
Iteration bounds the contraction factor by
$e^{-m\mu h}+2\ell^2h/\mu\le1/6+1/6=1/3$.

Let $z$ be the last iterate, evaluate $T(z)$, set $z^+=\Pi_{\mathcal K}(z-T(z)/\ell)$, and evaluate $T(z^+)$. Projection optimality gives
\[
 n^+=\ell(z-z^+)-T(z)\in N_{\mathcal K}(z^+),\qquad
 r=T(z^+)+n^+.
\]
The map $z\mapsto\Pi_{\mathcal K}(z-T(z)/\ell)$ is $2$-Lipschitz and fixes $z_*$. Hence $\norm{z-z^+}\le3\norm{z-z_*}$ and $\norm r\le2\ell\norm{z-z^+}\le6\ell\norm{z-z_*}$. Markov's inequality and the iterated epoch bound prove the result.
\end{proof}

\subsection{Accelerated proximal dual steps with an actual primal certificate}
Let $F:X\times Y\to\R$ be continuous, differentiable and convex--concave, where $X,Y$ are compact convex. Write $D_X=\operatorname{diam}(X)$, $D_Y=\operatorname{diam}(Y)$ and $d=D_X+D_Y>0$. Define the genuine convex dual objective
\[
 H(y)=-\min_{x\in X}F(x,y),\qquad y\in Y.
\]
Fix a proximal parameter $\tau>0$, a horizon $K\ge1$ and a residual tolerance $\zeta>0$. Initialize $y_0\in Y$, $z_0=y_0$ and $\mathsf A_0=0$. For $k=1,\ldots,K$, choose the positive $a_k$ satisfying
\begin{equation}\label{hd:eq:accelerated-weights}
 \tau a_k^2=\mathsf A_{k-1}+a_k=\mathsf A_k,
 \qquad v_k=\frac{\mathsf A_{k-1}y_{k-1}+a_kz_{k-1}}{\mathsf A_k}.
\end{equation}
Approximately solve the saddle subproblem
\begin{equation}\label{hd:eq:dual-prox-saddle}
 \min_{x\in X}\max_{y\in Y}
       \{F(x,y)-\tfrac\tau2\norm{y-v_k}^2\}.
\end{equation}
Specifically obtain $(x_k,y_k)$ and normal vectors $n_{x,k}\in N_X(x_k)$, $n_{y,k}\in N_Y(y_k)$ whose residual satisfies
\begin{equation}\label{hd:eq:dual-residual}
 \left\Vert
 \begin{pmatrix}
 \nabla_xF(x_k,y_k)+n_{x,k}\\
 -\nabla_yF(x_k,y_k)+\tau(y_k-v_k)+n_{y,k}
 \end{pmatrix}\right\Vert\le\zeta.
\end{equation}
Define $g_k=\tau(v_k-y_k)$, $z_k=z_{k-1}-a_kg_k$ and
$\bar x=\mathsf A_K^{-1}\sum_{k=1}^K a_kx_k$.
There is no projection of $z_k$ and no requirement that $v_k\in Y$.

\begin{lemma}[Accelerated primal--dual gap bound]\label{hd:lem:accelerated-primal}
If every residual check \eqref{hd:eq:dual-residual} holds, then
\begin{equation}\label{hd:eq:accelerated-gap}
 \max_{y\in Y}F(\bar x,y)-\min_{x\in X}F(x,y_K)
 \le\frac{D_Y^2}{2\mathsf A_K}+Kd\zeta
 \le\frac{2\tau D_Y^2}{(K+1)^2}+Kd\zeta.
\end{equation}
The left side of \eqref{hd:eq:accelerated-gap} is an actual primal--dual gap.  Since
$\min_xF(x,y_K)\le\min_x\max_yF(x,y)$, it implies
\[
 \max_yF(\bar x,y)
 \le \min_x\max_yF(x,y)+\frac{D_Y^2}{2\mathsf A_K}+Kd\zeta.
\]
Thus the accelerated procedure returns the primal objective certificate required by the outer proximal loop; a dual objective gap alone is never substituted for it.

In particular, for any $\delta>0$, the choices
\begin{equation}\label{hd:eq:dual-horizon}
 K=\max\{1,\lceil2\sqrt{\tau D_Y^2/\delta}\rceil\},\qquad
 \zeta=\frac{\delta}{2Kd}
\end{equation}
guarantee a primal saddle objective error at most $\delta$.
\end{lemma}
\begin{proof}
Let $e=d\zeta$. Convexity in $x$ and the first residual block give
$H(y_k)\le-F(x_k,y_k)+D_X\zeta$.
For every $y\in Y$, concavity in $y$ and the second residual block, with the normal-cone sign, give
\[
 -F(x_k,y)\ge-F(x_k,y_k)+\ip{g_k}{y-y_k}-D_Y\zeta.
\]
Combining,
\begin{equation}\label{hd:eq:affine-lower-model}
 \ell_k(y):=H(y_k)+\ip{g_k}{y-y_k}-e
 \le-F(x_k,y)\le H(y)\qquad(y\in Y).
\end{equation}
Thus the affine model is valid for the actual saddle point used for primal averaging, not only for the dual value function.

The identity \eqref{hd:eq:accelerated-weights} and $g_k=\tau(v_k-y_k)$ imply
\[
 \mathsf A_{k-1}(y_k-y_{k-1})
 =a_k(z_{k-1}-y_k)-a_k^2g_k.
\]
Using \eqref{hd:eq:affine-lower-model} at $y_{k-1}$, and expanding the squared distance in $z_k=z_{k-1}-a_kg_k$, yields for every fixed $y\in Y$
\begin{align*}
 &\mathsf A_kH(y_k)-\mathsf A_{k-1}H(y_{k-1})
   +\tfrac12\norm{z_k-y}^2-\tfrac12\norm{z_{k-1}-y}^2\\
 &\quad\le a_k[H(y_k)+\ip{g_k}{y-y_k}]
           +\mathsf A_{k-1}e-\tfrac12a_k^2\norm{g_k}^2\\
 &\quad\le-a_kF(x_k,y)+\mathsf A_ke.
\end{align*}
Sum and discard the terminal squared distance. Since $z_0\in Y$,
\[
 \mathsf A_KH(y_K)+\sum_{k=1}^Ka_kF(x_k,y)
 \le D_Y^2/2+\sum_{k=1}^K\mathsf A_ke
 \le D_Y^2/2+K\mathsf A_Ke.
\]
Convexity of $F(\cdot,y)$ gives the same inequality with the weighted sum replaced by $\mathsf A_KF(\bar x,y)$. Since $\min_xF(x,y_K)=-H(y_K)$, taking the maximum over $y$ proves the first inequality.
Finally, if $b_k=\tau\mathsf A_k$, then $\sqrt{b_k}=(1+\sqrt{1+4b_{k-1}})/2\ge\sqrt{b_{k-1}}+1/2$, and $b_1=1$. Thus $\mathsf A_K\ge(K+1)^2/(4\tau)$. The parameter choices follow.
\end{proof}

\subsection{A uniform compact primal region for proximal subproblems}
For a current primal center $c$ set
\[
 F_c(x,y)=f(x,y)+L\norm{x-c}^2.
\]
It is $L$-strongly convex in $x$ and concave in $y$ because the average $f$ is jointly $L$-smooth under AS. The components of $F_c$ satisfy mean-squared averaged smoothness with constant at most $3L$, by \cref{as:lem:moments}. They need not individually be $3L$-smooth. We now justify a finite primal ball without imposing an additional restriction on the original problem.

\begin{lemma}[Gradient bound from the original gap and dual diameter]\label{hd:lem:center-gradient}
If $\Phi(c)-\inf\Phi\le\Delta_c$, then for every $y\in Y$,
\begin{equation}\label{hd:eq:center-gradient}
 \norm{\nabla_x f(c,y)}\le3\sqrt{L\Delta_c}+LD_Y.
\end{equation}
Consequently every minimizer of $F_c(\cdot,y)$ lies within distance
$3\sqrt{\Delta_c/L}+D_Y$ of $c$. Restricting the minimization of $\max_yF_c(x,y)$ to a centered ball of any strictly larger radius does not change its optimum.
\end{lemma}
\begin{proof}
The maximum $\Phi$ is $L$-weakly convex. Let $p$ be its proximal point with coefficient $L$ and $g=2L(c-p)$. Comparison with the test point $c$ gives $\norm g\le2\sqrt{L\Delta_c}$. Proximal optimality and the maximum rule imply that $g$ is a convex combination of $\nabla_x f(p,y')$ over maximizing dual points $y'$ at $p$. One can apply the convex maximum rule to $\Phi(x)+L\norm x^2/2$ and subtract $Lp$ to see this without using a differentiable value function. Each such gradient differs from $\nabla_xf(c,y)$ by at most $L\norm{c-p}+LD_Y$. The triangle inequality proves \eqref{hd:eq:center-gradient}.

Strong monotonicity of $\nabla_xF_c(\cdot,y)$ bounds the distance of its zero from $c$ by $\norm{\nabla_xf(c,y)}/L$.
For completeness, along any ray from $c$, at radius $r$ the radial derivative of every $F_c(\cdot,y)$ is at least $Lr-\norm{\nabla_xf(c,y)}$. It is strictly positive beyond the asserted common radius. Radial projection onto a larger centered ball therefore decreases every $F_c(\cdot,y)$ outside that ball, and hence decreases their maximum. Thus the restriction does not change the primal optimum either.
\end{proof}

\subsection{Finite-sum implementation of every accelerated dual step}
On $X_c\times Y$, with $X_c$ the ball just justified, apply Lemma~\ref{hd:lem:svrg-residual} to the saddle operator of
\[
 F_c(x,y)-\frac\tau2\norm{y-v_k}^2,
 \qquad \tau=\frac L{\sqrt n}.
\]
The average operator is $\tau$-strongly monotone. To check its mean-squared smoothness, write
\[
 T_{i,c,v_k}(x,y)=G_i(x,y)+(2L(x-c),\tau(y-v_k)).
\]
The common added map has Lipschitz constant $2L$, since $\tau\le L$.
Thus \cref{as:lem:moments} gives
\[
 \left(\frac1n\sum_i\norm{T_{i,c,v_k}(z)-T_{i,c,v_k}(z')}^2\right)^{1/2}
 \le3L\norm{z-z'}\le4L\norm{z-z'}.
\]
Use the conservative value $\ell=4L$ in the solver; no individual
$4L$-Lipschitz bound is asserted. To verify strong monotonicity without
differentiating twice, add the two strong convexity and strong concavity
inequalities at one point to bound its saddle-operator inner product
against another point, then do the same with the points interchanged.
The cross function values cancel, leaving
$L\norm{x-x'}^2+\tau\norm{y-y'}^2$.
The components need not be monotone individually. The normal-cone domain is the product $X_c\times Y$, whose projection is available from the known sets.

Since $(4L/\tau)^2=16n$, each SVRG epoch costs $O(n)$ calls. The residual criterion is exactly \eqref{hd:eq:dual-residual}. The normal components are obtained from the explicit projected-step formula, not from an additional normal-cone oracle. The inner problem is on a compact set of known diameter, irrespective of the magnitude of the extrapolated $v_k$; hence its SVRG distance bound is uniform in $v_k$.
It follows that one original primal proximal problem is solved to objective error $\delta$ in
\begin{equation}\label{hd:eq:prox-cost-informal}
 O\!\left(n\left[1+\sqrt{\frac{LD_Y^2}{\sqrt n\,\delta}}\right]
       \left[1+\log_{+}\frac{(4L)d}{\zeta\sqrt p}\right]\right)
\end{equation}
component calls, with a computable residual check at every accelerated step. Here $d=\operatorname{diam}(X_c)+D_Y$, $K,\zeta$ are as in \eqref{hd:eq:dual-horizon}, and $p$ is a chosen failure probability per accelerated step. On all successful checks the objective-error conclusion is deterministic by Lemma~\ref{hd:lem:accelerated-primal}.

\subsection{Outer descent, public budgets and success probability}
We include the inexact proximal calculation. Suppose $x_{j+1}$ satisfies
\begin{equation}\label{hd:eq:outer-inexact}
 \Phi(x_{j+1})+L\norm{x_{j+1}-x_j}^2
 \le\min_x\{\Phi(x)+L\norm{x-x_j}^2\}+\delta.
\end{equation}
Let $p_j$ be the exact proximal point. Strong convexity gives $\norm{x_{j+1}-p_j}^2\le2\delta/L$, while comparison with $x_j$ gives
$L\norm{x_{j+1}-x_j}^2\le\Phi(x_j)-\Phi(x_{j+1})+\delta$.
Consequently
\begin{equation}\label{hd:eq:outer-descent}
 \frac1S\sum_{j=0}^{S-1}\norm{\nabla\Phi_{1/(2L)}(x_j)}^2
 \le\frac{8L\Delta}{S}+24L\delta.
\end{equation}
This uses only the original weakly convex value function, not a dual-smoothed substitute.

Here are public parameter choices, making the entire algorithm fixed-budget:
\begin{equation}\label{hd:eq:upper-parameters}
 \begin{split}
 \delta&=\epsilon^2/(288L),\quad S=\max\{1,\lceil96A\rceil\},\quad
 \tau=L/\sqrt n,\\
 R_x&=3\sqrt{(\Delta+S\delta)/L}+2D_Y+\epsilon/L,
 \quad d=2R_x+D_Y,\\
 K&=\max\{1,\lceil2\sqrt{\tau D_Y^2/\delta}\rceil\},\quad
 \zeta=\delta/(2Kd),\quad p=1/(12SK).
 \end{split}
\end{equation}
At each center use the ball of radius $R_x$ centered there and run the $K$ accelerated dual steps. At every such step run
\[
 J=\max\left\{1,\left\lceil
       \frac{\log(36(4L)^2d^2/(p\zeta^2))}{\log3}
       \right\rceil\right\}
\]
SVRG epochs, then compute the full residual with two extra passes. Abort and use an arbitrary output if a residual check fails. After completing all $S$ primal steps, output a uniformly sampled center among $x_0,\ldots,x_{S-1}$, independently of the preceding random choices.

Whenever the algorithm has not aborted, all completed outer steps obey \eqref{hd:eq:outer-inexact}; in particular $\Phi(x_j)-\inf\Phi\le\Delta+j\delta$. Lemma~\ref{hd:lem:center-gradient} therefore justifies the same $R_x$ at every subsequent step. Each residual failure has conditional probability at most $p$, uniformly over the prior successful history. A union bound over at most $SK$ checks bounds the abort probability by $1/12$.
On a nonaborted run, \eqref{hd:eq:outer-descent} is at most $\epsilon^2/6$, so the independent uniform output fails the Moreau threshold with probability at most $1/6$. Total failure is at most $1/4<1/3$.

Finally $S=O(1+A)$, $K=O(1+B/n^{1/4})$, and
\[
 \Delta+S\delta=O(\Delta+\epsilon^2/L),\qquad
 \frac{Ld^2}{\delta}=O(1+A+B^2).
\]
Substitution in the residual budget gives $J=O(\log(2+n+A+B))$. Each accelerated dual step has cost $O(nJ)$. Thus the fixed total number of component calls is at most
\[
 C(1+A)(n+n^{3/4}B)\log(2+n+A+B),
\]
proving \eqref{eq:accelerated-main}. This upper proof is for arbitrary instances of the original class, not only for the quadratic inner modules of the lower construction.

\section{A Cone Construction for the Small-Diameter Lower Bound}
\label{app:cone}
This appendix proves \eqref{eq:new-lower}. Its normalized smoothness
constant $L_0$, scalar profiles, and block variables are local to this
construction. The physical parameters are those of \cref{sec:model-main}.
The independent smooth baseline needed for bounded $B$ is established
in \cref{sec:moreau-facts}.

\subsection{The normalized construction}
\label{sec:construction}

\subsubsection{Numerical profiles and radial protection}
Fix a nonnegative \(C^\infty\) density \(k\), supported in \((0,1)\),
with integral one and supremum at most two. Such a density is obtained
by convolving \(2\mathbf 1_{[1/4,3/4]}\) with a sufficiently narrow
nonnegative smooth probability mollifier. Put
\[
\omega(t)=\int_{-\infty}^t k(s)\,ds,\qquad
H(t)=\omega(4t-1),\quad g(t)=\omega(t-1),\quad
\theta(t)=\int_0^t\omega(s)\,ds.
\]
Here \(\theta=0\) for \(t\le0\). These functions have the properties
\begin{align*}
&H=0\ \text{on }(-\infty,1/4],\qquad
  H=1\ \text{on }[1/2,\infty),\\
&g=0\ \text{on }(-\infty,1],\qquad
  g=1\ \text{on }[2,\infty),\qquad 0\le g'\le2,\\
&\theta=0\ \text{on }(-\infty,0],\qquad
  \theta'=1\ \text{on }[1,\infty),\qquad
  0\le\theta'\le1,\quad 0\le\theta(t)\le t_+ .
\end{align*}
All derivative bounds used below are numerical constants.
Since \(g\) is smooth and takes values in \([0,1]\), \(g'=0\)
at every point where \(g\) equals zero or one.

Define a smooth nondecreasing radius profile \(\varphi\) by
\(\varphi(r)=r\) for \(r\le2\),
\(\varphi'(r)=1-\omega(r-2)\) for \(2\le r\le3\), and
\(\varphi'=0\) for \(r\ge3\). Set
\[
\zeta(z)=
\begin{cases}
\varphi(\norm z)\,z/\norm z,&z\ne0,\\
0,&z=0.
\end{cases}
\]
Then
\begin{equation}\label{eq:radial}
\zeta(z)=z\quad(\norm z\le2),\qquad
\norm{\zeta(z)}\le3,\qquad
\norm{D\zeta(z)}_{\op}\le1,\qquad
\norm{D^2\zeta(z)}_{\op}\le K_\zeta
\end{equation}
for a universal \(K_\zeta\). Indeed, the radial and tangential
eigenvalues of \(D\zeta\) are \(\varphi'\) and \(\varphi(r)/r\),
both in \([0,1]\); differentiating once more gives the
dimension-free bound
\(\norm{D^2\zeta}\le|\varphi''|+4|\varphi'-\varphi/r|/r\)
outside \(r=2\). The map is the identity near zero.

Let \(h_1=\norm{H'}_\infty\) and choose a fixed
\[
M_0>2(1+40h_1).
\]
A smooth radial guard \(J(z)=j(\norm z)\) is defined by \(j=0\)
on \(r\le3/2\), and
\[
j''(r)=4M_0\omega(4(r-3/2)),\qquad
j'(r)=\int_{3/2}^rj''(s)\,ds,\qquad
j(r)=\int_{3/2}^rj'(s)\,ds
\]
above that radius. It is nonnegative, \(j'\ge0\),
\(j'(r)\ge M_0\) for \(r\ge2\), and
\(\norm{\nabla^2J}\le4M_0\).

\subsubsection{Hidden sources and the two primal variables}
Take
\begin{equation}\label{eq:normalized-parameters}
R\ge1,\qquad q=\frac{\sqrt n}{R},\qquad
q\ge16,\qquad q\ge n^{1/4}.
\end{equation}
There are \(M\ge1\) stages. For each \(j\), independently choose a
Haar orthonormal \(n\)-frame
\((w_{1j},\ldots,w_{nj})\) in \(\R^d\), where \(d\ge n\)
will be fixed in the information proof. Define
\begin{equation}\label{eq:hidden-source}
v_j=\frac1{\sqrt n}\sum_{i=1}^n w_{ij},\qquad \norm{v_j}=1.
\end{equation}
Stage \(j\) has primal variables \(z_j,u_j\in\R^d\) and
dual variable \(y_j\in\R^d\). The vector \(z_j\) is a unit-scale
source variable; \(u_j\) is a clock of scale \(q\).

For a unit \(v\), write
\[
t=\ip u v,\qquad p=u-tv,\qquad
s=\sqrt{1+\norm u^2},\qquad s_p=\sqrt{1+\norm p^2}.
\]
Define the clock and its exit gate by
\begin{align}
C_v(z,u)
&=s-1+\frac{\norm u^2}{2q}
 -10\theta\!\left(t-s_p+4H(\ip{\zeta(z)}v)\right),
\label{eq:clock}\\
G_v(u)
&=g\!\left(\frac{t-\sqrt{q^2+\norm p^2}}q\right).
\label{eq:exit-gate}
\end{align}
The source latch is added \emph{inside} the argument of \(\theta\).
It is not multiplied by the \(q\)-scale clock potential.

\subsubsection{Actual component functions}
Set \(G_0=1\) and \(G_j=G_{v_j}(u_j)\). The known dual domain is
\[
Y_R=\left\{y=(y_1,\ldots,y_M):\sum_j\norm{y_j}^2\le R^2\right\}.
\]
For \(i\in[n]\), define the fixed smooth function
\begin{equation}\label{eq:components}
\boxed{\quad
f_i(z,u,y)=
\sum_{j=1}^M\bigl[C_{v_j}(z_j,u_j)+J(z_j)\bigr]
+\frac1R\sum_{j=1}^MG_{j-1}
  \ip{\zeta(z_j)-\sqrt n\,w_{ij}}{y_j}.
\quad}
\end{equation}
The average is affine in \(y\), and its value is exactly
\begin{equation}\label{eq:value}
\Phi(z,u)=\sum_{j=1}^M[C_{v_j}(z_j,u_j)+J(z_j)]+S(z,u),
\qquad
S=\left(\sum_{j=1}^MG_{j-1}^2
                \norm{\zeta(z_j)-v_j}^2\right)^{1/2}.
\end{equation}
This identity follows directly from
\(n^{-1}\sum_i\sqrt n\,w_{ij}=v_j\) and maximization over
the Euclidean ball. No unconstrained-dual elimination is used.

\subsection{Global admissibility}
\label{sec:smoothness}
\begin{lemma}[Uniform individual smoothness and gap]\label{lem:admissibility}
There is a universal \(L_0\ge1\), independent of
\(n,R,q,M,d\) and the frame realization, such that every component
\eqref{eq:components} has jointly \(L_0\)-Lipschitz gradient on
\(\R^{2Md}\times Y_R\). Moreover,
\[
\Phi(0)=1,\qquad
\Phi(0)-\inf\Phi\le1+91Mq.
\]
\end{lemma}
\begin{proof}
Let \(h_2=\norm{H''}_\infty\), \(t_2=\norm{\theta''}_\infty\).
For
\(a_v=t-s_p+4H(\ip{\zeta(z)}v)\), the composition rule gives
\[
\norm{\nabla a_v}^2\le2+16h_1^2,\qquad
\norm{\nabla^2a_v}\le
\max\{1,4(h_2+h_1K_\zeta)\}.
\]
Thus \(\norm{\nabla^2C_v}\le L_C\) for a universal \(L_C\).
The first derivative of the quadratic clock need not be bounded;
only its Hessian is being bounded.

If \(K_{G,1}=\sqrt2\norm{g'}_\infty\) and
\(K_{G,2}=2\norm{g''}_\infty+\norm{g'}_\infty\), then
\[
\norm{\nabla G_v}\le K_{G,1}/q,\qquad
\norm{\nabla^2G_v}\le K_{G,2}/q^2.
\]
For one source term
\(\ell_{ij}=R^{-1}G_{j-1}\ip{\zeta(z_j)-\sqrt n w_{ij}}{y_j}\),
the nonzero Hessian block norms are bounded by
\[
\begin{array}{c|c@{\qquad}c|c}
\text{block}&\text{bound}&\text{block}&\text{bound}\\ \hline
uu&K_{G,2}(\sqrt n+3)/q^2&
uy&K_{G,1}(\sqrt n+3)/(Rq)\\
uz&K_{G,1}/q&zz&K_\zeta\\
zy&1/R&yy&0 .
\end{array}
\]
Here \(\norm{y_j}\le R\) throughout the specified domain.
The first term omits \(u_0\).
Conditions \eqref{eq:normalized-parameters} imply
\[
\frac{\sqrt n+3}{Rq}\le4,\qquad
\frac{\sqrt n+3}{q^2}\le4,\qquad R^{-1}\le1.
\]
Different \(\ell_{ij}\) use disjoint variable groups
\((u_{j-1},z_j,y_j)\). Their Hessian sum is therefore block
diagonal after a permutation of coordinates. The clock-and-guard
sum is separately block diagonal in \((z_j,u_j)\).
Bounding a local block matrix by the scalar matrix of its block norms
gives, for example, the universal admissible choice
\[
L_0=\max\{1,L_C+4M_0+4K_{G,2}+5K_{G,1}+K_\zeta+2\}.
\]
Integrating the Hessian on segments in the convex product domain proves
the joint Lipschitz-gradient inequality.

At the origin, all clocks and guards vanish, every \(G_j\), \(j\ge1\),
is zero, and the sole active residual is \(-v_1\). Hence \(\Phi(0)=1\).
Since
\(\theta(a_v)\le(a_v)_+\le\norm u+4\),
\[
C_v(z,u)\ge \frac{\norm u^2}{2q}-10\norm u-41
           \ge-50q-41\ge-91q.
\]
The guards and \(S\) are nonnegative, which proves the gap bound.
\end{proof}

\subsection{Every small subgradient recovers every source}
\label{sec:stationarity}
Fix \(\delta=10^{-3}\). Write
\(r_j=\zeta(z_j)-v_j\) and
\(\mathcal R=(G_{j-1}r_j)_{j=1}^M\), so \(S=\norm{\mathcal R}\).
Every Clarke subgradient \(w\) of \(\Phi\) has the form
\begin{equation}\label{eq:subgradient}
w=\nabla\sum_j(C_{v_j}+J)+D\mathcal R^\top a,\qquad
\begin{cases}
a=\mathcal R/S,&S>0,\\
\norm a\le1,&S=0.
\end{cases}
\end{equation}
This follows either from the norm-composition rule or directly from
Danskin's formula for the compact dual ball. It also contains every
subgradient needed by proximal optimality.

\begin{lemma}[Guard and gate separation]\label{lem:binary}
If \(w\in\partial_C\Phi(z,u)\) and \(\norm w\le\delta\), then
\(\norm{z_j}<2\) and \(G_j\in\{0,1\}\) for every \(j\).
\end{lemma}
\begin{proof}
The source contribution to a \(z_j\)-block in \eqref{eq:subgradient}
has norm at most one. The clock's \(z_j\)-gradient has norm at most
\(40h_1\). At \(\norm{z_j}\ge2\), the guard's radial derivative
is at least \(M_0>2(1+40h_1)\), making a small subgradient impossible.
Thus \(\zeta(z_j)=z_j\) at the point in question.

Suppose \(0<G_j<1\). With \(r=\norm p\) and
\(\xi=(t-\sqrt{q^2+r^2})/q\), one has \(1<\xi<2\), and
\[
t-\sqrt{1+r^2}+4H\ge
q+\sqrt{q^2+r^2}-\sqrt{1+r^2}\ge q\ge1.
\]
Hence \(\theta'=1\).
Only the next source term can add to the \(u_j\)-gradient.
It adds \(\alpha_j\nabla G_j\), where \(|\alpha_j|\le4\):
for \(S>0\),
\(\alpha_j=G_j\norm{r_{j+1}}^2/S\le\norm{r_{j+1}}\le4\);
for \(S=0\) and a partial gate, \(r_{j+1}=0\);
for \(j=M\), the term is absent.
The perpendicular part of \(w_{u_j}\) is consequently
\[
\left(\frac1s+\frac1q+\frac{10}{\sqrt{1+r^2}}
-\frac{\alpha_jg'(\xi)}{q\sqrt{q^2+r^2}}\right)p.
\]
For \(q\ge16\) its coefficient is at least
\(10/\sqrt{1+r^2}+1/(2q)\). Therefore
\[
\frac{10r}{\sqrt{1+r^2}}\le\delta,\qquad
r\le\frac{\delta}{\sqrt{100-\delta^2}}<0.001.
\]
It follows that \(t/q=\sqrt{1+r^2/q^2}+\xi<3.01\).
The parallel component, however, satisfies
\[
\ip{w_{u_j}}v
=\frac ts+\frac tq-10+\frac{\alpha_jg'(\xi)}q
\le1+3.01-10+\frac8q<-5,
\]
contradicting \(\norm w\le\delta\).
\end{proof}

\begin{lemma}[Zero residual and clock completion]\label{lem:all-sources}
Under the hypotheses of \cref{lem:binary}, \(S=0\), and
\[
z_j=v_j,\qquad H(\ip{z_j}{v_j})=1,\qquad G_j=1
\quad\text{for every }j.
\]
\end{lemma}
\begin{proof}
All gates are binary. If \(S>0\), let
\(\mathcal A=\{j:G_{j-1}=1\}\). On an active block the source
gradient is \(r_j/S=(z_j-v_j)/S\).
The clock correction is \(-\tau_jv_j\), where
\[
\tau_j=40\theta'(a_{v_j})H'(\ip{z_j}{v_j})\ge0.
\]
If \(\tau_j>0\), then \(1/4<\ip{z_j}{v_j}<1/2\); hence
\(\ip{r_j}{-\tau_jv_j}\ge0\).
The guard also has nonnegative inner product with \(r_j\), since it is
radial and starts only at radius \(3/2>\norm{v_j}\).
Taking the inner product of the full \(z\)-gradient with
\((r_j/S)_{j\in\mathcal A}\), a unit vector, gives at least one.
Thus \(S>0\) contradicts \(\norm w\le\delta\).

We next verify a uniform clock fact:
\begin{equation}\label{eq:clock-fact}
H=1,\qquad \norm{\nabla_u C_v}\le0.1
\quad\Longrightarrow\quad G_v(u)=1.
\end{equation}
For \(H=1\), put \(a=t-s_p+4\). Then
\[
\nabla_u C_v=
\left(\frac ts+\frac tq-10\theta'(a)\right)v+
\left(\frac1s+\frac1q+\frac{10\theta'(a)}{s_p}\right)p.
\]
If \(t<0\), the cone force has nonnegative inner product with \(u\);
therefore \(\norm{\nabla_uC_v}\ge\norm u/s\).
A norm at most \(0.1\) would give \(\norm u<0.101\), and then
\(a>2.89\), \(\theta'=1\), and the negative parallel component
has magnitude at least ten, a contradiction. Thus \(t\ge0\).
If \(a<1\), then \(t<s_p-3\) and \(s_p>3\).
Since \(s<\sqrt2\,s_p\), this yields
\(\norm p/s>2/3\), again inconsistent with a small gradient.
Hence \(\theta'=1\). The perpendicular and parallel equations give
\[
\norm p\le\frac{0.1}{\sqrt{100-0.1^2}},
\qquad \frac tq\ge8.9.
\]
For \(q\ge1\), the argument of \(g\) is therefore greater than two,
proving \eqref{eq:clock-fact}.

Finally \(G_0=1\). Since \(S=0\), this forces \(z_1=v_1\)
and \(H_1=1\). If \(G_1=0\), then \(\nabla G_1=0\) and the
full \(u_1\)-subgradient is exactly the clock gradient, contradicting
\eqref{eq:clock-fact}. Thus \(G_1=1\).
Inductively \(G_{j-1}=1\) forces \(z_j=v_j\), \(H_j=1\) and
\(G_j=1\). This also covers the terminal clock.
\end{proof}

\begin{corollary}[Original Moreau certificate]\label{cor:moreau-source}
If \(\norm{\nabla\Phi_{1/(2L_0)}(\widehat x)}\le\delta\), then
for every \(j\),
\begin{equation}\label{eq:output-source}
\norm{\widehat z_j-v_j}\le\frac{\delta}{2L_0}.
\end{equation}
\end{corollary}
\begin{proof}
Individual smoothness implies that \(\Phi\) is \(L_0\)-weakly convex.
Its proximal minimizer \(x^*\) is unique and satisfies
\[
2L_0(\widehat x-x^*)=\nabla\Phi_{1/(2L_0)}(\widehat x)
\in\partial\Phi(x^*).
\]
The preceding lemmas imply \(z_j^*=v_j\).
Since \(\norm{x^*-\widehat x}\le\delta/(2L_0)\),
\eqref{eq:output-source} follows. The proximal point is not being
treated as an algorithm query.
\end{proof}

\subsection{Unrestricted adaptive information cost}
\label{sec:information}
The information proof uses conditionally uniform orthonormal frames. The proof checks the exact values and joint gradients of the concrete
components \eqref{eq:components}.

\subsubsection{A uniform conditional spherical estimate}
For one stage, after \(k\) distinct rows have been revealed,
conditional on those rows and all independent public information,
\begin{equation}\label{eq:haar-decomposition}
v=v_{\rm known}+\sqrt{1-k/n}\,U,\qquad
\norm{v_{\rm known}}=\sqrt{k/n},
\end{equation}
where \(U\) is uniform on the unit sphere in the orthogonal complement
of the revealed rows, of dimension \(d-k\).
This follows from the rotational invariance of the unrevealed
orthonormal frame. It remains true for adaptive row indices:
an unexposed column has the same conditional law irrespective of
its label.

Let
\[
\tau=\frac1{24},\qquad
k_0=\left\lfloor\frac n{9216}\right\rfloor.
\]
For \(k<k_0\) and any history-measurable vector \(h\) with
\(\norm h\le1\), \(\norm{v_{\rm known}}\le\tau/4\).
The elementary spherical-cap inequality yields
\begin{equation}\label{eq:cap-bound}
\Prob\bigl(|\ip h v|>\tau\mid\text{ideal history}\bigr)
\le2\exp\!\left(-\frac{(d-n-2)\tau^2}{8}\right).
\end{equation}
For completeness, a uniform sphere vector \(U\in\R^m\) has moments
\[
\E U_1^{2\ell}
=\frac{(2\ell-1)!!}{m(m+2)\cdots(m+2\ell-2)}
\le\frac{(2\ell-1)!!}{m^\ell}.
\]
Odd moments vanish. Rotating a vector of norm at most one onto the first
coordinate and summing its exponential series gives
\(\E e^{t\ip hU}\le e^{t^2/(2m)}\).
Chernoff's inequality gives
\(\Prob(|\ip hU|>a)\le2e^{-ma^2/2}\).
Project \(h\) onto the unrevealed subspace; its norm is still at most one.
The known contribution is at most \(\tau/4\), so a violation requires
the spherical contribution to exceed \(\tau/2\) in absolute value.
This proves the relaxed bound \eqref{eq:cap-bound}.
It is uniform in the norm of the original query after normalization.

\subsubsection{The ideal oracle and its coupling}
Fix any algorithm with at most \(Q<Mk_0\) calls, including its private
seed as part of the ideal history. A stage is completed when \(k_0\)
distinct rows have been granted. Completed stages form a prefix, because
only the earliest unfinished stage receives rows.
An ideal experiment operates as follows.

\begin{enumerate}[leftmargin=*,itemsep=3pt]
\item Freeze the completed prefix for the current reply.
Frames of completed stages are recorded in full in the analyst's
enlarged filtration. For each unfinished
stage, censor its hidden part of \(C_v\), retaining only
\(\sqrt{1+\norm u^2}-1+\norm u^2/(2q)\), and set its gate \(G_v\)
and the gate's first derivatives to zero.
\item Internally reveal the selected row of the earliest unfinished
stage, even if its coefficient in the reply is zero. Repeated indices
do not increase the row count. Form the value and gradient reply of the
censored formula using that row and the frozen completed prefix.
\item After forming the reply, if that stage reaches \(k_0\) distinct
rows, grant its entire frame for free. This affects subsequent
queries, not the reply just formed.
\end{enumerate}

These internal reveals are included in the enlarged filtration for
analysis; the original algorithm receives only the computed value and
gradient replies. Its next query is nevertheless measurable with respect
to that larger filtration. Thus the internal reveals introduce no
extra restriction on the algorithm and no artificial observation into
its original-format transcript.
There is at most one unfinished stage receiving a row on any call.
Thus fewer than \(Mk_0\) calls leave at least one unfinished stage.

For every ideal query, and for the final ideal output, test every stage
that is unfinished at that time against the two vectors of norm at most one
\[
h_z=\zeta(z)/3,\qquad
h_u=
\begin{cases}u/\norm u,&u\ne0,\\0,&u=0.\end{cases}
\]
They are measurable before the relevant fresh row is granted.
Using \eqref{eq:cap-bound} and a union bound, the probability of any
test with \(|\ip{h_z}v|>\tau\) or
\(|\ip{h_u}v|>\tau\) is at most
\begin{equation}\label{eq:bad-probability}
4M(Q+1)\exp\!\left(-\frac{(d-n-2)\tau^2}{8}\right).
\end{equation}
Choose the finite dimension
\begin{equation}\label{eq:dimension}
d\ge n+3+\frac8{\tau^2}\log(24M(Q+1)).
\end{equation}
Then \eqref{eq:bad-probability} is at most \(1/6\).
One may use \(Q=Mk_0-1\) in this dimension choice once and for all,
thereby covering every smaller budget.

On the complementary event, an unfinished stage has
\(\ip{\zeta(z)}v\le3\tau=1/8\), so \(H=H'=0\).
Moreover \(\ip{u/\norm u}v\le\tau<1/\sqrt2\).
The hidden clock can be active with \(H=0\) only if
\[
\ip u v>\sqrt{1+\norm{u-\ip u v\,v}^2},
\]
which would require normalized correlation greater than \(1/\sqrt2\).
The exit gate can be positive only inside an even narrower cone.
Hence the clock's hidden value and full first jet vanish, and
\(G=0,\nabla G=0\), exactly as in the ideal formula.

This checks \emph{all} hidden dependence in the true reply:
the uncensored source term of the earliest unfinished stage is
\[
\frac{G_{\rm prev}}R
\ip{\zeta(z)-\sqrt n\,w_i}{y},
\]
whose value and gradients use only its selected row and
completed-stage data. It contains no additional hidden mean \(v\).
Later source terms have both zero gate and zero gate gradient.
The true and ideal transcripts therefore coincide inductively on the
good event, including threshold queries, arbitrary real arithmetic
and the final output. The cap bounds were taken under the ideal
history without conditioning on earlier good events; thus no
unjustified conditional Haar assertion is used.

\subsubsection{Contradiction at a successful output}
An unfinished stage still exists when \(Q<Mk_0\). On the good event,
its output satisfies
\(\ip{\zeta(\widehat z)}v\le1/8\).
In contrast, \cref{cor:moreau-source} would give
\(\norm{\widehat z-v}\le\delta/(2L_0)<0.001\).
Then \(\norm{\widehat z}<2\), \(\zeta(\widehat z)=\widehat z\), and
\(\ip{\zeta(\widehat z)}v>0.999\).
Thus success is impossible on the good event.

\begin{theorem}[Fixed-budget normalized lower bound]\label{thm:normalized-lower}
For sufficiently large universal \(n\), the static distribution of
\eqref{eq:components}, in the finite dimension \eqref{eq:dimension},
has the following property. Every unrestricted randomized algorithm
using fewer than \(Mk_0\) queries has probability at most \(1/6\)
of producing a \(\delta\)-stationary point of
\(\Phi_{1/(2L_0)}\). In particular its worst-case fixed-budget
complexity is \(\Omega(Mn)\).
\end{theorem}
The distribution is fixed before the algorithm runs. The argument
allows the ideal query sequence to depend on the entire transcript
and the private seed, so it directly covers randomized algorithms.
A success guarantee of \(2/3\) on every instance is incompatible with
average success at most \(1/6\) on this distribution.

\subsection{Scaling and the small-gap regime}
\label{sec:scaling}
First take \(n\ge65536\), which ensures \(q\ge n^{1/4}\ge16\)
and \(k_0\ge n/18432\). Smaller \(n\) are handled below.
Let
\[
B_0=\frac{2L_0}{\delta},\qquad K_0=\frac{L_0}{\delta^2}.
\]
First suppose \(B\ge B_0\) and
\(R=B/B_0\le n^{1/4}\). Set
\[
\sigma=\frac{\eps L_0}{L\delta},\qquad
\alpha=\frac{\eps^2L_0}{L\delta^2},\qquad
q=\frac{\sqrt n}{R}=\frac{B_0\sqrt n}{B}.
\]
Use the physical functions and domain
\[
f_i^{\rm phys}(x,y)=\alpha f_i(x/\sigma,y/\sigma),
\qquad Y_{\rm phys}=\sigma Y_R.
\]
Their individual smoothness is at most
\(\alpha L_0/\sigma^2=L\), and their dual diameter is
\(2\sigma R=D_Y\). A physical query is simulated by one normalized
query. The exact Moreau scaling is
\[
\Phi^{\rm phys}_{1/(2L)}(\sigma x)
=\alpha\Phi_{1/(2L_0)}(x),\qquad
\nabla\Phi^{\rm phys}_{1/(2L)}(\sigma x)
=\frac{\eps}{\delta}\nabla\Phi_{1/(2L_0)}(x).
\]
It follows by substituting \(z=\sigma Z\) in the proximal minimization;
no smooth-value/Moreau comparison is needed for this construction.

The actual dimensionless initial gap is at most
\[
A_{\rm actual}\le K_0(1+91Mq).
\]
If \(A\ge400K_0q\), take
\[
M=\left\lfloor\frac{A}{200K_0q}\right\rfloor.
\]
Then \(A_{\rm actual}\le A\),
\(M\ge A/(400K_0q)\), and \cref{thm:normalized-lower} gives
\[
\T\ge cMn\ge
\frac{c}{400K_0B_0}\,A\sqrt n\,B.
\]
If \(A<400K_0q\), then
\begin{equation}\label{eq:small-A}
A\sqrt n\,B<400K_0B_0\,n.
\end{equation}
Thus the desired term is already covered, up to a universal constant,
by the independent hidden-component \(\Omega(n)\) instance.
This covers the small-gap case at the same universal accuracy threshold.

For \(B<B_0\), \(\sqrt n\,\barB\le B_0\sqrt n\), so the
smooth finite-sum baseline suffices. For every \(B\ge B_0\), including
the capped case, choose
\[
R=\min\{B/B_0,n^{1/4}\},\qquad q=\sqrt n/R,\qquad
F=\min\{\sqrt n B,n^{3/4}\}.
\]
Using a smaller diameter than the permitted \(D_Y\) is allowed.
One has \(n/q=\sqrt nR\ge F/B_0\). If \(A\ge400K_0q\),
the same choice of \(M\) gives a lower bound proportional to
\(AF/(400K_0B_0)\). If \(A<400K_0q\), then
\[
AF\le B_0 A n/q<400K_0B_0 n,
\]
so the additive lower bound again suffices.
This proves the capped coefficient in \eqref{eq:new-lower}.
The finitely many \(n\) below the fixed threshold needed for
\(q\ge16\) and \(k_0=\Omega(n)\) are absorbed by the same baseline:
\(n^{3/4}/\sqrt n=n^{1/4}\) is bounded on that finite range.

Finally, for \(B\le n^{1/6}\),
\[
\barB^4\le \sqrt n\,\barB,\qquad
\min\{\sqrt n\,\barB,n^{3/4}\}=\sqrt n\,\barB.
\]
Combining \eqref{eq:new-lower} with \eqref{eq:general-upper},
and using \(A\ge A_0\) to replace \(1+A\) by a constant multiple of \(A\),
proves the first line of \eqref{global:eq:piecewise}.

\section{A Serial Branch Construction for the High-Diameter Lower Bound}
\label{app:branch}\label{hd:sec:result}
We first prove a matrix-information lower bound and then realize it
with a fixed smooth finite sum. All constants and profiles in this
appendix are independent of those in \cref{app:cone}. The scalar
endpoint coupling is motivated by Pan et al.~\cite{pzl}; its vector
extension and every analytic property used below are proved here.

\begin{theorem}[High-diameter bounds]\label{hd:thm:main}
There are numerical constants $A_0,B_0,c,C>0$ with the following properties.
The upper bound \eqref{eq:accelerated-main} holds for all parameters.
It holds under AS by \cref{hd:sec:upper}.
If $A\ge A_0$ and $B\ge B_0 n^{1/4}$, then
\begin{equation}\label{hd:eq:matching}
 c n^{3/4}AB\le\T\le C n^{3/4}AB\log(2+n+A+B).
\end{equation}
Thus the high-diameter region is matching up to a logarithm. Neither the lower bound nor this statement imposes a span, support, norm, or output restriction on algorithms. The lower-bound dimension is finite but can grow exponentially in $M(N-1)$. This high-diameter theorem alone makes no matching assertion for the complementary parameter region; the all-regime theorem also uses the cone and replication arguments.
\end{theorem}

\subsection{Orthogonal branches and a source-release oracle}
Fix integers $n,H\ge1$. For a multiplicity $D$, define
\[
 \calH_t=(\R^n)^{\otimes t}\otimes\R^D,\quad d_t=n^tD,
 \quad 0\le t\le H.
\]
Let $e\in\R^D$ be a fixed unit vector and $s=n^{-1/2}(1,\ldots,1)$. Write
\[
 \xi_t=s^{\otimes t},\qquad
 J_{t,i}a=e_i\otimes a,\qquad J_{t,s}=n^{-1/2}\sum_iJ_{t,i}.
\]
Then $J_{t,i}^{\mathsf T}J_{t,j}=\mathbf1_{i=j}I$ and $J_{t,s}$ is an isometry. Sample independent Haar rotations $O_t\in O(d_t)$ for $t\ge1$ and a Haar rotation $O_0$ fixing $e$. All are sampled before interaction. Set
\begin{equation}\label{hd:eq:maps}
 A_{t,i}=O_{t+1}J_{t,i}O_t^{\mathsf T},\quad
 U_t=O_{t+1}J_{t,s}O_t^{\mathsf T},\quad
 u_t=O_t(\xi_t\otimes e).
\end{equation}
Thus $u_0=e$, $U_tu_t=u_{t+1}$, and
\[
 \norm{A_{t,i}}_\op=1,\qquad
 n^{-1}\sum_iA_{t,i}=n^{-1/2}U_t.
\]

\subsubsection{The matrix part of a round}
At each round $k$, the algorithm chooses on every edge arbitrary vectors $a_t,b_t$ and an index $i_{t,k}\in[n]$. All inputs precede all replies. The oracle returns
\begin{equation}\label{hd:eq:matrix-oracle}
 \bigl(A_{t,i_{t,k}}a_t,\ A_{t,i_{t,k}}^{\mathsf T}b_t\bigr)_{0\le t<H}.
\end{equation}
Indices may differ across edges. Query norms, histories, and computations are unrestricted.

\subsubsection{A rule for releasing source vectors}
Fix a threshold $a\in(0,1)$ and put $\rho=a^2/2$. Fix a public set of source sites $\mathscr S\subseteq\{1,\ldots,H\}$, of size at most $M\le H$. A source at a site can be released once.

We maintain word sets $\Lambda_t\subseteq[n]^t$, initialized by $\Lambda_0=\{\varnothing\}$ and $\Lambda_t=\varnothing$ for $t>0$. In a matrix round update them simultaneously by
\begin{equation}\label{hd:eq:forward-update}
 \widetilde\Lambda_t=\Lambda_t\cup
 \{(i_{t-1,k},w):w\in\Lambda_{t-1}\}\quad(t\ge1),
 \qquad \widetilde\Lambda_0=\{\varnothing\}.
\end{equation}
Let $p_t=|\Lambda_t|/n^t$. After the update, a source $u_j$ may be released when $j\in\mathscr S$ is unreleased and $p_j\ge\rho$. Release reveals the entire physical vector $u_j$ and updates the \emph{accounting sets} by
\begin{equation}\label{hd:eq:reset}
 \Lambda_t\leftarrow[n]^t\quad(0\le t\le j).
\end{equation}
It does not reveal bases or rotations at earlier layers. Repeat eligible releases in any order. Giving all eligible sources immediately is allowed. The accounting sets, release flags, and indices may be public.

This is a stateful, enlarged matrix oracle, not itself an oracle for an arbitrary finite sum. It is useful because an exact-flat nonlinear component query can be simulated by it under the coupling in \cref{hd:sec:flat}. The release rule depends only on the index/accounting history and previously released vectors, never on an unexposed part of a rotation.

\subsection{A release costs a bounded factor of reachable mass}
\begin{lemma}[Deletion closure and one-reset amplification]\label{hd:lem:reset}
At every time, if $(i,w)\in\Lambda_{t+1}$ then $w\in\Lambda_t$. Consequently $p_0\ge p_1\ge\cdots\ge p_H$. If a legal release takes place, every coordinate of the mass vector after the release is at most $\rho^{-1}$ times its value before the release.
\end{lemma}
\begin{proof}
Deletion closure holds initially. In \eqref{hd:eq:forward-update}, an inserted word has a suffix in the old lower set, hence also in the new lower set. Existing words retain their old suffixes. A reset fills all layers up to its site and therefore preserves closure, including across the interface to the next layer.

There are at most $n$ words for each suffix, so $|\Lambda_{t+1}|\le n|\Lambda_t|$. If site $j$ is released, then $p_t\ge p_j\ge\rho$ for $t\le j$. Replacing these entries by one multiplies each by at most $\rho^{-1}$. Other entries do not change.
\end{proof}

\begin{theorem}[Adaptive chronological accounting with source releases]\label{hd:thm:accounting}
After $K$ matrix rounds and $r_K$ source releases,
\begin{equation}\label{hd:eq:mass-bound}
 p_t(K)\le\min\left\{1,\rho^{-r_K}\frac{\binom Kt}{n^t}\right\},
 \qquad 0\le t\le H.
\end{equation}
The inequality holds pathwise for every adaptive index schedule and every allowed choice of release times. A binomial coefficient is zero when $t>K$.
\end{theorem}
\begin{proof}
Between releases the simultaneous update gives
\[
 \widetilde p_t\le p_t+n^{-1}p_{t-1}.
\]
A release increases the whole vector by at most a common factor $\rho^{-1}$, by Lemma~\ref{hd:lem:reset}. Induct on rounds and insert the release operations after each round. The induction step for a round is Pascal's identity,
\[
 \frac{\binom{k-1}t}{n^t}+\frac1n
 \frac{\binom{k-1}{t-1}}{n^{t-1}}=\frac{\binom kt}{n^t}.
\]
At $t=0$ the upper bound is at least one. Multiplying by $\rho^{-1}$ for every release gives \eqref{hd:eq:mass-bound}; finally truncate at one. No independence is used in this counting argument.
\end{proof}

\begin{remark}
A source can inject all $n^j$ source coordinates at layer $j$, so the original no-release support theorem cannot simply be quoted unchanged. The factor $\rho^{-r_K}$ is precisely the new charge. Since only $M$ sources are released while a path has $H$ edges, this factor can be dominated by the binomial propagation deficit. Caching before a source is released has not been discarded.
\end{remark}

\subsection{Exact adaptive product--Haar coupling, including source reveals}

We use the following finite-dimensional disintegration fact repeatedly.  It
is stated explicitly because the same rotation is shared by several primal
and dual blocks at a macro endpoint.

\begin{lemma}[Conditional Haar extension of a partial isometry]
\label{hd:lem:haar-extension}
Let $E,F$ be real Euclidean spaces of the same finite dimension.  Let
$K\subseteq E$, $L\subseteq F$, and let $V:K\to L$ be an isometry onto
$L$.  Write $\mathfrak E(V)$ for the orthogonal maps $O:E\to F$ satisfying
$O|_K=V$, equipped with the conditional Haar law obtained by choosing an
independent Haar isometry $K^\perp\to L^\perp$.

Suppose that, conditional on a sigma-field $\mathcal F$, $O$ has law
$\mathfrak E(V)$ and that $K,L,V$ are $\mathcal F$-measurable.
\begin{enumerate}[label=(\roman*),itemsep=3pt]
\item Let $x\in F$ be $\mathcal F$-measurable.  Put
$x_0=P_Lx$, $x_\perp=x-x_0$.  If $x_\perp\ne0$, let
$g\sim N(0,I_E)$ be independent of $\mathcal F$ and set
\[
 h=\frac{(I-P_K)g}{\|(I-P_K)g\|},\qquad
 K'=K\oplus\operatorname{span}\{h\},\qquad
 L'=L\oplus\operatorname{span}\{x_\perp\},
\]
with
\[
 V'|_K=V,\qquad V'h=x_\perp/\|x_\perp\|.
\]
Then
\[
 O^{\mathsf T}x=V^{\mathsf T}x_0+\|x_\perp\|h,
\]
and, conditional on $\mathcal F$ and this revealed inverse image, $O$ has
law $\mathfrak E(V')$.  If $x_\perp=0$, the same statement holds with
$(K',L',V')=(K,L,V)$ and $O^{\mathsf T}x=V^{\mathsf T}x$.

\item Let $y\in E$ be $\mathcal F$-measurable.  Put
$y_0=P_Ky$, $y_\perp=y-y_0$.  If $y_\perp\ne0$, draw a unit vector
$z$ uniformly from the sphere of $L^\perp$, independently of
$\mathcal F$, and extend $V$ by
\[
 V'(y_\perp/\|y_\perp\|)=z.
\]
Then
\[
 Oy=Vy_0+\|y_\perp\|z,
\]
and, conditional on the revealed forward image, $O$ has law
$\mathfrak E(V')$.  The zero-residual case again leaves $V$ unchanged.

\item The preceding statements remain valid when the choice between a
forward reveal, an inverse reveal, or no reveal, and the vector to be
revealed, are stopping-time decisions measurable with respect to the
current augmented history.  For a finite product of independent rotations,
performing one such reveal at a time preserves conditional independence of
the unexposed completions.
\end{enumerate}
All assertions are understood outside the null event on which a projected
Gaussian in a nonzero complement vanishes.
\end{lemma}

\begin{proof}
Choose orthonormal identifications of $K^\perp$ and $L^\perp$ with
$\mathbb R^r$.  Every $O\in\mathfrak E(V)$ is uniquely $V\oplus Q$, where
$Q$ is Haar on $O(r)$.  If $r=0$, both assertions reduce to their
zero-residual cases and are immediate.  Assume $r\ge1$ and fix a unit
$u\in L^\perp$.  Consider the following
two-stage construction: draw $h$ uniformly from the unit sphere of
$K^\perp$; conditional on $h$, draw a Haar isometry
$R:h^\perp\to u^\perp$ and define $Qh=u$, $Q|_{h^\perp}=R$.  For any fixed
orthogonal map $B$ of $K^\perp$, the vector $B^{\mathsf T}h$ is again
uniform.  Conditional on $B^{\mathsf T}h$, the map $QB$ sends
$B^{\mathsf T}h$ to $u$, while its restriction to
$(B^{\mathsf T}h)^\perp$ is obtained from $R$ by precomposition with the
fixed isometry $B$.  That restriction is therefore Haar.  Hence $QB$ has the
same law as $Q$ for every $B$.  The constructed probability law on the
compact group $O(r)$ is right invariant, and thus is normalized Haar
measure.  Since the construction also satisfies $Q^{\mathsf T}u=h$
deterministically, its second-stage kernel is a regular conditional law of a
Haar $Q$ given $Q^{\mathsf T}u=h$.  (When $r=1$, the complementary fiber is
a singleton.)  This proves (i), since a normalized projected standard
Gaussian is uniform on the sphere.  Applying the same argument to
$Q^{\mathsf T}$ proves (ii).

For (iii), condition on the current augmented history and apply (i) or
(ii).  The next decision is then fixed, and the updated conditional law is
again a Haar completion of the enlarged partial isometry.  Induction proves
the stopping-time statement.  In a product law, a reveal involving one
factor changes only that factor's partial isometry; the other factors retain
their independent conditional Haar laws.  Sequentially revealing several
factors therefore preserves the asserted product structure.
\end{proof}

Fix integers $b,b_{\rm out}\ge1$.  A \emph{$b$-channel round} is a matrix
round together with, at each vertex $t$, an ordered list of at most $b$
physical vectors measurable before any reply in that round.  The forward
input on edge $t$ and the transpose input on edge $t-1$ are included in this
list whenever present.  Marking the additional vectors returns no
information to the interacting algorithm; it is only an accounting device.
At termination, at most $b_{\rm out}$ further physical vectors per vertex
may be marked.

\begin{lemma}[Lazy product--Haar rotations with adaptive channels]
\label{hd:lem:lazy}
Fix a deterministic adaptive strategy using at most $K$ $b$-channel rounds,
the release rule above, and at most $b_{\rm out}$ marked terminal vectors
per vertex.  There is a coupling with independent standard Gaussian vectors
\[
 g_{t,k,r}\sim N(0,I_{d_t})
 \quad(0\le t\le H,\ 1\le k\le K,\ 1\le r\le b),
\]
and
\[
 g^{\rm out}_{t,r}\sim N(0,I_{d_t})
 \quad(0\le t\le H,\ 1\le r\le b_{\rm out}),
\]
with the following properties.
\begin{enumerate}[label=(\roman*),itemsep=3pt]
\item The completed rotations have exactly the independent product--Haar
law in \eqref{hd:eq:maps}; at the root the completion is Haar in the
subgroup fixing $e$.
\item Whenever a marked physical vector $x\in\cal H_t$ is processed in its
assigned slot, its formal coordinate $O_t^{\mathsf T}x$ belongs to the span
of the previously exposed formal vectors at layer $t$ and the Gaussian
assigned to that slot.
\item Every matrix reply is exactly
\[
 O_{t+1}J_{t,i}O_t^{\mathsf T}a=A_{t,i}a,
 \qquad
 O_tJ_{t,i}^{\mathsf T}O_{t+1}^{\mathsf T}b=A_{t,i}^{\mathsf T}b.
\]
A released source is exactly $O_t(\xi_t\otimes e)=u_t$ and uses no formal
Gaussian slot.
\item All marked vectors in a round are processed before any matrix reply.
The terminal marked vectors are processed after the last release closure.
Unused Gaussian slots remain independent and may be included in the later
reservoir count.
\end{enumerate}
\end{lemma}

\begin{proof}
At every vertex maintain subspaces $K_t\subseteq\cal H_t$ in formal
coordinates and $L_t\subseteq\cal H_t$ in physical coordinates, together
with an isometry $V_t:K_t\to L_t$.  Initially
$K_0=L_0=\operatorname{span}\{e\}$ and $V_0e=e$; for $t\ge1$ the three
objects are zero-dimensional.  The induction invariant is:
conditional on the entire augmented history, the unexposed completions of
all $O_t$ are independent and $O_t$ is a Haar completion of $V_t$.

Process the marked physical vectors in the public lexicographic order
$(t,r)$.  Apply Lemma~\ref{hd:lem:haar-extension}(i), using the Gaussian
assigned to the slot.  This exposes the exact inverse image and enlarges
$V_t$ only when the physical residual is nonzero.  The formula in that
lemma proves property (ii), including linearly dependent or repeated marked
vectors.

After all marked inputs have been processed, the formal coordinates of all
matrix-action inputs are known.  For a forward action, form the known formal
vector $J_{t,i}O_t^{\mathsf T}a$ and expose its image through $O_{t+1}$ by
Lemma~\ref{hd:lem:haar-extension}(ii).  For a transpose action do the same
with $J_{t,i}^{\mathsf T}O_{t+1}^{\mathsf T}b$ through $O_t$.  Process the
finitely many actions in a fixed order.  Sequential conditioning gives the
joint law of all replies, not merely their separate marginals, and proves
(iii).

After the matrix part of a round, apply the public accounting update and
repeat all eligible releases.  A release decision is measurable from the
augmented history and the public accounting sets.  Expose the deterministic
formal vector $\xi_j\otimes e$ forward through $O_j$ using
Lemma~\ref{hd:lem:haar-extension}(ii).  The stopping-time clause shows that
the invariant is preserved.  The same inverse-reveal operation handles the
terminal marked vectors.  Finally complete every $V_t$ by an independent
conditional Haar isometry.  The induction invariant then gives exactly the
static product--Haar marginal in \eqref{hd:eq:maps}; no rotation is chosen
in response to an unexposed future query.
\end{proof}

For the remainder of the information argument, use all the Gaussian slots
in Lemma~\ref{hd:lem:lazy}, including unused ones, and set
\begin{equation}\label{hd:eq:reservoir}
 \begin{aligned}
 S_H&=\sum_{t=0}^H n^t,
 &m_G&=(bK+b_{\rm out})S_H,\\
 V&=\Span\{\text{all last-factor slices of all formal Gaussians}\}
 \subseteq\mathbb R^D.
 \end{aligned}
\end{equation}
There are exactly $m_G$ independent standard Gaussian slices.  If
$D>m_G$, their span has rank $m_G$ almost surely and is rotationally
invariant, so
\begin{equation}\label{hd:eq:projection}
 \mathbb E\|P_Ve\|^2=m_G/D.
\end{equation}

\begin{lemma}[Support with releases and the correlation estimate]
\label{hd:lem:support}
At every atomic time after the corresponding accounting update, every
exposed formal vector at layer $t$ lies in
\[
 \mathcal R_t:=
 \bigl(E_t\otimes\Span\{e\}\bigr)+
 \bigl(\mathbb R^{n^t}\otimes V\bigr),
 \qquad E_t=\Span\{e_w:w\in\Lambda_t\}.
\]
This includes all marked query vectors, all matrix replies, all released
sources, and all marked terminal vectors.  Consequently, if a physical
vector $z$ satisfies $\|z\|\le1$ and
$O_t^{\mathsf T}z\in\mathcal R_t$, then
\begin{equation}\label{hd:eq:correlation}
 |\langle u_t,z\rangle|^2
 \le p_t+(1-p_t)\|P_Ve\|^2
 \le p_t+\|P_Ve\|^2.
\end{equation}
\end{lemma}

\begin{proof}
Proceed through the same atomic order as in
Lemma~\ref{hd:lem:lazy}.  An inverse reveal adds a vector in the span of the
old formal exposed space and one assigned Gaussian.  The latter belongs to
$\mathbb R^{n^t}\otimes V$, so it contributes only to the nuisance summand.
Both nuisance spaces and their orthogonal sums over words are invariant
under every $J_{t,i}$ and $J_{t,i}^{\mathsf T}$.

For the source part, a forward action prepends the selected index, exactly
as in \eqref{hd:eq:forward-update}.  A transpose action either annihilates a
source word or deletes its first letter; deletion closure from
Lemma~\ref{hd:lem:reset} places the resulting suffix in the lower-layer
set.  Before a source is exposed, the reset fills all word sets through its
site, so $\xi_j\otimes e\in E_j\otimes\Span\{e\}$.  These observations prove
$O_t^{\mathsf T}z\in\mathcal R_t$ for every exposed vector.  Enlarging an
accounting set can only enlarge $\mathcal R_t$, so the statement also covers
vectors exposed earlier in the same round.

Put $v=\|P_Ve\|^2$.  Orthogonally decompose the last factor as
$V\oplus V^\perp$.  The first summand of $\mathcal R_t$ adds, beyond the
unrestricted $V$ part, precisely
$E_t\otimes\Span\{P_{V^\perp}e\}$.  Therefore the squared norm of the
projection of the unit target $\xi_t\otimes e$ onto $\mathcal R_t$ is
\[
 v+(1-v)\|P_{E_t}\xi_t\|^2
 =v+(1-v)\frac{|\Lambda_t|}{n^t}
 =v+(1-v)p_t.
\]
Cauchy--Schwarz applied to $O_t^{\mathsf T}z$ proves
\eqref{hd:eq:correlation}.
\end{proof}

\subsection{Serial query complexity in the enlarged oracle}
\begin{theorem}[Serial source recovery with adaptive local channels]
\label{hd:thm:serial}
Fix $a\in(0,1)$, put $\rho=a^2/2$, and suppose that at most
$M\le H$ source sites are available.  Fix integers $b,b_{\rm out}\ge1$ and
set
\begin{equation}\label{hd:eq:dimensions}
 K_* =\left\lfloor\frac{\rho^2 nH}{4\mathrm e}\right\rfloor,
 \qquad
 D=\left\lceil
 \frac{32(bK_*+b_{\rm out})}{a^2}\sum_{t=0}^Hn^t
 \right\rceil.
\end{equation}
For the fixed product--Haar law in these finite dimensions, every adaptive
randomized strategy using at most $K_*$ matrix rounds, at most $b$ marked
physical vectors per vertex and round, and at most $b_{\rm out}$ marked
terminal vectors per vertex satisfies the following: for every marked
terminal vector $\widehat v\in\cal H_H$ with $\|\widehat v\|\le1$,
\begin{equation}\label{hd:eq:terminal-failure}
 \mathbb P\bigl(|\langle u_H,\widehat v\rangle|<a\bigr)\ge15/16.
\end{equation}
The statement allows arbitrary cached probes and all source vectors released
by the public rule.  Conversely, $nH$ matrix rounds suffice to recover
$u_H$ exactly.  Hence, for fixed $a,b,b_{\rm out}$, the success-$2/3$
terminal-direction complexity is $\Theta_a(nH)$.
\end{theorem}

\begin{proof}
First fix the private seed, so the strategy is deterministic, and apply
Lemma~\ref{hd:lem:lazy}.  Use the $K_*$-round reservoir even when the actual
number of rounds is smaller; all unused Gaussian slots remain independent.
Thus
\[
 m_G=(bK_*+b_{\rm out})\sum_{t=0}^Hn^t,
 \qquad
 \mathbb E\|P_Ve\|^2\le a^2/32.
\]
With the event
\[
 \mathcal E=\{\|P_Ve\|^2\le\rho\},
\]
Markov's inequality gives $\mathbb P(\mathcal E^c)\le1/16$.

Let $K\le K_*$ be the actual number of rounds and let $r_K\le M$ be the
number of releases.  If $K<H$, Theorem~\ref{hd:thm:accounting} gives
$p_H=0$.  If $K\ge H$, then
\begin{align*}
 p_H
 &\le \rho^{-r_K}\frac{\binom KH}{n^H}
 \le \rho^{-M}\left(\frac{\mathrm eK_*}{nH}\right)^H\\
 &\le \rho^{-M}(\rho^2/4)^H
 \le(\rho/4)^H
 \le\rho/4.
\end{align*}
The third inequality uses the definition of $K_*$; the fourth uses
$M\le H$ and $0<\rho<1$.  The same conclusion therefore holds in both
cases.  On $\mathcal E$, Lemma~\ref{hd:lem:support} yields, simultaneously
for every marked terminal vector of norm at most one,
\[
 |\langle u_H,\widehat v\rangle|^2
 \le p_H+\|P_Ve\|^2
 \le\rho/4+\rho
 =5a^2/8<a^2.
\]
This proves \eqref{hd:eq:terminal-failure} for the fixed seed.  Integrating
over an independent private seed preserves the same static product--Haar
instance law.

For the converse, $u_0=e$ is public.  Given $u_t$, query every selected map
on the same vector and average:
\[
 u_{t+1}=\frac1{\sqrt n}\sum_{i=1}^nA_{t,i}u_t.
\]
Repeating over all $H$ edges uses $nH$ rounds and returns $u_H$ exactly.
\end{proof}

\begin{corollary}[Macro-stage interpretation]
For $M$ macro stages with $m\ge1$ branch edges per stage, take $H=Mm$ and
permit one source release at each macro endpoint.  For any fixed finite
numbers of local and terminal channels, the enlarged oracle requires
$\Omega_a(nMm)$ rounds.  This is a chronological product lower bound, not a
sum of separate one-stage bounds.
\end{corollary}

\begin{remark}[Role of the local channels]
The parameter $b$ counts all physical vectors at a vertex whose formal
coordinates must be controlled before the reply, including the two possible
matrix-action inputs.  Increasing $b$ or $b_{\rm out}$ changes only the
finite multiplicity $D$ in \eqref{hd:eq:dimensions}; it does not change the
round lower bound.  No extra matrix action is granted by a marked channel.
\end{remark}

\subsection{A branching two-port saddle with the required radius scale}\label{hd:sec:pipe}
Fix $N\ge2$. Take $N$ consecutive vertex spaces and isometries $U_t$ as above, now indexed $0,\ldots,N-1$. Write $P=U_{N-2}\cdots U_0$. The primal connectors are $a\in\calH_0$ and $b\in\calH_{N-1}$; dual blocks are $w_t\in\calH_t$.

\subsubsection{Scalar normalization, proved explicitly}
Let $L_N$ be the Neumann path Laplacian (endpoint diagonal one, interior diagonal two, and adjacent entries minus one). Define
\begin{equation}\label{hd:eq:scalar-M}
 \lambda_N=N^{-2},\quad M_N=\lambda_N I+L_N,\quad
 k_N=(M_N^{-1})_{0,N-1},\quad
 c_N=\frac{(M_N^{-1})_{0,0}}{k_N},\quad d_N=\frac{2-c_N}{2}.
\end{equation}
\begin{lemma}[Uniform endpoint constants]\label{hd:lem:scalar}
For every $N\ge2$,
\[
 N/2<k_N\le N,\qquad 1\le c_N<\cosh1<8/5,
 \qquad 1/5<d_N\le1/2.
\]
\end{lemma}
\begin{proof}
Put $\vartheta=2\operatorname{arsinh}(1/(2N))$, so $2\cosh\vartheta-2=N^{-2}$ and $0<N\vartheta<1$. The last column of $M_N^{-1}$ is
\[
 x_j=C\cosh((j+1/2)\vartheta),\quad
 C=\frac1{2\sinh(N\vartheta)\sinh(\vartheta/2)},
 \quad j=0,\ldots,N-1.
\]
The root equation follows from the fictitious boundary $x_{-1}=x_0$, the interior equations from the hyperbolic recurrence, and the last equation from
$\cosh((N+1/2)\vartheta)-\cosh((N-1/2)\vartheta)=2\sinh(N\vartheta)\sinh(\vartheta/2)$.
By reflection symmetry the two diagonal endpoint Green coefficients agree. Consequently
\[
 k_N=\frac{N\cosh(\vartheta/2)}{\sinh(N\vartheta)}>N/\sinh1>N/2,
 \qquad c_N=\frac{\cosh((N-1/2)\vartheta)}{\cosh(\vartheta/2)}<\cosh1.
\]
All $x_j\ge x_0=k_N$. Since $\mathbf1^\mathsf T M_N=N^{-2}\mathbf1^\mathsf T$, their sum is $N^2$, giving $Nk_N\le N^2$. Finally $c_N\ge1$ and the elementary bound $\cosh1<8/5$ give the claimed interval for $d_N$.
\end{proof}

Set
\begin{equation}\label{hd:eq:pipe-params}
 \beta=n^{-1/2},\qquad r_N=\sqrt{\beta/k_N}.
\end{equation}
For each selected component define the global quadratic
\begin{align}
 \mathscr H_i(a,b;w)={}&d_N(\norm a^2+\norm b^2)
 +r_N(\ip a{w_0}-\ip b{w_{N-1}})\notag\\
 &-\frac\beta2\left[\lambda_N\sum_{t=0}^{N-1}\norm{w_t}^2
       +\sum_{t=0}^{N-2}(\norm{w_t}^2+\norm{w_{t+1}}^2)\right]
 +\sum_{t=0}^{N-2}\ip{w_{t+1}}{A_{t,i}w_t}.
 \label{hd:eq:pipe-components}
\end{align}
Its exact average is
\begin{equation}\label{hd:eq:pipe-average}
 \mathscr H=d_N(\norm a^2+\norm b^2)
 +r_N(\ip a{w_0}-\ip b{w_{N-1}})
 -\frac\beta2\left[\lambda_N\sum_t\norm{w_t}^2+
          \sum_{t<N-1}\norm{w_{t+1}-U_tw_t}^2\right].
\end{equation}

\begin{theorem}[Smooth components, exact endpoint energy, and dual radius]\label{hd:thm:pipe}
For the quadratics above:
\begin{enumerate}[label=(\roman*),itemsep=3pt]
\item Each component has globally $8$-Lipschitz joint gradient, uniformly in $n,N,D$. The average is $\beta/N^2$-strongly concave in $w$.
\item For a self-adjoint operator $\mathcal D_N$ on the terminal connector space,
\begin{equation}\label{hd:eq:Q}
 Q(a,b):=\max_w\mathscr H(a,b;w)
 =\norm a^2-\ip b{Pa}+\ip b{\mathcal D_Nb},
 \qquad d_N I\preceq\mathcal D_N\preceq I,\quad \mathcal D_NP=P.
\end{equation}
In particular
\begin{equation}\label{hd:eq:Q-bounds}
 \frac15(\norm a^2+\norm b^2)< Q(a,b)
 \le\frac32(\norm a^2+\norm b^2)
\end{equation}
for nonzero $(a,b)$, and $Q$ has a dimension-independent positive Hessian lower bound. On the aligned source line $a=a_0u_0$, $b=b_0u_{N-1}$, it is exactly $a_0^2-a_0b_0+b_0^2$.
\item The unique unconstrained dual maximizer obeys
\begin{equation}\label{hd:eq:dual-radius}
 \norm{w^*(a,b)}\le\sqrt3\,Nn^{1/4}
                 \sqrt{\norm a^2+\norm b^2}.
\end{equation}
\item A selected-component value--gradient query is exactly simulated by one set of branch actions on the internal edges and public arithmetic. No hidden source vector occurs in the formula for this inner oracle.
\end{enumerate}
\end{theorem}

\begin{proof}
Let $\mathcal M$ denote the operator in the dual quadratic bracket of \eqref{hd:eq:pipe-average}. It is at least $N^{-2}I$. Its smallest eigenvalue is exactly $N^{-2}$, as can be seen on a vector with $w_{t+1}=U_tw_t$.
For individual smoothness, the diagonal Hessian norm is at most three, the symmetric off-diagonal dual part has norm at most two, and the endpoint primal--dual coupling part has norm at most $2r_N\le2$ (the last inequality uses $k_N>N/2\ge1$). Thus seven, and hence eight, bounds the full joint Hessian norm. Components need not be dual-concave; the average is, as required by the original model.

We compute the endpoint blocks of $\mathcal M^{-1}$. Define innovation spaces $\mathcal K_0=\calH_0$ and $\mathcal K_j=\ker U_{j-1}^{\mathsf T}$ for $j\ge1$. The orthogonal decomposition
\[
 \calH_t=\bigoplus_{j=0}^t U_{t-1}\cdots U_j\mathcal K_j
\]
separates the dual operator into scalar suffix chains. The $j=0$ chain is $M_N$. For $j>0$ it is the trailing principal submatrix of $M_N$, equivalently a shorter Neumann chain with the same mass plus a unit entrance anchoring term. Consequently
\begin{equation}\label{hd:eq:green}
 (\mathcal M^{-1})_{00}=c_Nk_N I,\qquad
 (\mathcal M^{-1})_{N-1,0}=k_NP,
\end{equation}
and, writing $G=(\mathcal M^{-1})_{N-1,N-1}$,
\begin{equation}\label{hd:eq:terminal-green}
 0\prec G\preceq c_Nk_N I,\qquad GP=c_Nk_NP.
\end{equation}
For the inequality, the inverse of a positive principal submatrix is at most the corresponding principal block of the full inverse. This follows immediately from the positive Schur-complement formula. Apply it to each shorter scalar suffix to bound its terminal Green coefficient by that of $M_N$. Orthogonality of the innovation decomposition yields the operator bound.

Let $E(a,b)$ insert $a$ into layer zero and $-b$ into layer $N-1$. Completing the dual square gives
\[
 Q=d_N(\norm a^2+\norm b^2)
       +\frac1{2k_N}\ip{E(a,b)}{\mathcal M^{-1}E(a,b)}.
\]
Equations \eqref{hd:eq:green}--\eqref{hd:eq:terminal-green} give \eqref{hd:eq:Q} with
$\mathcal D_N=d_NI+G/(2k_N)$, because $d_N+c_N/2=1$.
Decompose $b=Pc+b_\perp$, with $P^{\mathsf T}b_\perp=0$. Self-adjointness and $\mathcal D_NP=P$ imply
\[
 Q=\norm a^2-\ip a c+\norm c^2+\ip{b_\perp}{\mathcal D_Nb_\perp}.
\]
The first three terms are between one half and three halves of $\norm a^2+\norm c^2$. Combine this with $d_N>1/5$ and $\mathcal D_N\preceq I$ to prove \eqref{hd:eq:Q-bounds}. The scalar restriction follows from $Pu_0=u_{N-1}$.

The dual maximizer is
\[
 w^*=\frac1{\sqrt{\beta k_N}}\mathcal M^{-1}E(a,b).
\]
Since $\mathcal M\succeq N^{-2}I$, spectral calculus gives
$\mathcal M^{-2}\preceq N^2\mathcal M^{-1}$. Its endpoint block, by \eqref{hd:eq:green}--\eqref{hd:eq:terminal-green}, has norm at most $k_N(c_N+1)$. It follows that
\[
 \norm{w^*}^2\le\frac{N^2}{\beta k_N}
 \ip{E(a,b)}{\mathcal M^{-1}E(a,b)}
 \le\frac{(c_N+1)N^2}{\beta}(\norm a^2+\norm b^2)
 <3N^2n^{1/2}(\norm a^2+\norm b^2).
\]
Finally, in \eqref{hd:eq:pipe-components} all hidden matrices occur only in local bilinear terms. Their values and both gradients are computed by the forward and transpose actions once per edge. The endpoint couplings and diagonal terms are public.
\end{proof}

\begin{lemma}[Connector coercivity before removing the dual constraint]\label{hd:lem:bounded-dual}
Take any number of these modules, with connector tuple $\nu=((a_j,b_j))_j$, and constrain their full dual tuple to one closed ball of radius $R_Y$. Define
\[
 Q_{R_Y}(\nu)=\max_{\norm w\le R_Y}\sum_j\mathscr H^{(j)}(a_j,b_j;w^{(j)}).
\]
This function is differentiable and at least $2/5$-strongly convex, with $\nabla Q_{R_Y}(0)=0$. If the full outer value function is
$\Phi(s,\nu)=G(s,\nu)+Q_{R_Y}(\nu)$, then
\begin{equation}\label{hd:eq:connector-coercivity}
 \norm\nu\le\frac52\left(\norm{\nabla_\nu\Phi(s,\nu)}+
                              \norm{\nabla_\nu G(s,\nu)}\right).
\end{equation}
In particular, a bound $\norm{\nabla_\nu G}\le C\sqrt{\Lambda(s)}$ gives a connector bound independent of the number of modules whenever the activation load $\Lambda(s)$ is bounded independently of that number. None of these statements assumes that the unconstrained dual maximizer is feasible.
\end{lemma}
\begin{proof}
Write $Q_{R_Y}=\sum_jd_N\norm{\nu_j}^2+q_{R_Y}(\nu)$. The second term is a maximum of affine functions of $\nu$ and is convex. At $\nu=0$, the unique dual maximizer is $w=0$. Strict dual concavity gives uniqueness for every $\nu$, so differentiation of the maximum gives a continuous gradient and $\nabla q_{R_Y}(0)=0$. The quadratic term has Hessian at least $2/5$ times the identity. Strong monotonicity therefore gives
\[
 \frac25\norm\nu^2\le\ip{\nabla Q_{R_Y}(\nu)}\nu.
\]
Substitute $\nabla Q_{R_Y}=\nabla_\nu\Phi-\nabla_\nu G$ and apply Cauchy--Schwarz, dividing by $\norm\nu$ unless it is zero. This proves \eqref{hd:eq:connector-coercivity}. After, and only after, a uniform connector bound has been established, \eqref{hd:eq:dual-radius} can be applied to show that the unconstrained maximizer fits into a suitably chosen ball of radius $O(Nn^{1/4})$.
\end{proof}

\subsection{Scalar primitives and a dimension-free radial cap}\label{hd:sec:primitives}
Define the smooth monotone step
\[
 \omega(t)=\begin{cases}
 0,&t\le0,\\
 \displaystyle\frac{e^{-1/t}}{e^{-1/t}+e^{-1/(1-t)}},&0<t<1,\\
 1,&t\ge1.
 \end{cases}
\]
It is flat at both endpoints, satisfies $\omega(t)+\omega(1-t)=1$, and
\begin{equation}\label{hd:eq:step-derivatives}
 0\le\omega'\le8,\qquad |\omega''|\le128.
\end{equation}
For completeness, on $0<t\le1/2$ put $u=1/t\ge2$. The logistic derivative is
$\omega'=\omega(1-\omega)[t^{-2}+(1-t)^{-2}]$ and
$\omega(1-\omega)\le e^{2-u}$. This bounds the first derivative by
$e^2e^{-u}(u^2+4)\le8$. Differentiating again bounds the absolute second derivative by
$e^2e^{-u}(u^4+2u^3+8u^2+32)<128$, using the separate maxima of the four monomials times $e^{-u}$ on $[2,\infty)$. Symmetry covers the other half interval.

Let $p(t)=0$ for $t\le0$ and $p(t)=\int_0^t\omega(v)\,dv$ for $t>0$. Then $p(t)=t-1/2$ for $t\ge1$. Set
\begin{equation}\label{hd:eq:scalar-gates}
 q(t)=\omega((5t-1)/4),\qquad
 e_s(t)=p(t+3)-p(t-2)-5/2.
\end{equation}
Thus $q=0$ on $(-\infty,1/5]$, $q=1$ on $[1,\infty)$, $0\le q'\le10$, and $|q''|\le200$.
Also $e_s(t)=t$ on $[-2,2]$, $|e_s|\le5/2$, $0\le e_s'\le1$, $|e_s''|\le16$, and $e_s'$ is zero outside $[-3,3]$.

For $C>0$, define the radial length function
\[
 f_C(r)=\begin{cases}
 r,&0\le r\le C,\\
 C+C\displaystyle\int_0^{(r-C)/C}(1-\omega(v))\,dv,&C<r<2C,\\
 3C/2,&r\ge2C.
 \end{cases}
\]
On any Euclidean space put
\begin{equation}\label{hd:eq:radial-cap}
 \rho_C(x)=\begin{cases} f_C(\norm x)x/\norm x,&x\ne0,\\0,&x=0.\end{cases}
\end{equation}
This map is smooth, is the identity for $\norm x\le C$, and is a scalar multiple of its input. It satisfies
\begin{equation}\label{hd:eq:cap-properties}
 \norm{\rho_C(x)}\le\min\{\norm x,2C\},\quad
 \norm{D\rho_C}\le1,\quad \norm{D^2\rho_C}\le18/C.
\end{equation}
Indeed $0\le f_C'\le1$, $f_C\le r$, $f_C\ge C$ for $r\ge C$, and $|f_C''|\le8/C$. For $g(r)=f_C(r)/r$, direct differentiation bounds $\norm{D^2(g(\norm x)x)}$ by $3|g'|+r|g''|\le |f_C''|+7/r$ outside the identity region. This proves the last assertion with room in the constant. Flatness of $\omega$ gives smooth gluing. In particular, for a unit $u$, the score $t=\ip u{\rho_C(x)}$ is $1$-Lipschitz.

\subsection{An explicit global outer construction}\label{hd:sec:global}
Fix integers $M\ge1$ and $N\ge2$, and put $m=N-1$ and $H=Mm$. Use the global branch spaces and rotations at depths $0,\ldots,H$. Stage $j$ uses depths $(j-1)m,\ldots,jm$. Its connectors are $a_j\in\calH_{(j-1)m}$, $b_j\in\calH_{jm}$; its dual tuple is a separate copy of these $N$ vertex spaces. Different stages have distinct dual variables, even at a shared endpoint depth. Its state is $S_j\in\calH_{jm}$. All vector blocks at the same depth use the same rotation. Write $v_j=u_{jm}$ and $v_0=u_0=e$.

The following constants are numerical and do not depend on any problem parameter:
\begin{equation}\label{hd:eq:outer-constants}
 C_\nu=21,\quad \sigma=42,\quad \gamma=4300,\quad
 K_s=172,\quad R_s=10,\quad
 \delta_g=\frac1{430000},\quad c_R=32416,\quad \tau_0=\frac1{20}.
\end{equation}
Let $E_\nu=\rho_{21}$, $Z=\rho_{10}$ and set
\begin{equation}\label{hd:eq:state-corrections}
 \begin{split}
 t_j&=\ip{v_j}{Z(S_j)},\qquad t_0=1,\\
 h(t)&=t\,\omega((t-\delta_g)/\delta_g),\\
 V(t)&=-\frac{c_R+1}{10}\,[p(10t-1)-p(10t-10)],\\
 P_j(S_j)&=\frac\gamma2\norm{S_j}^2-
              \frac\gamma2 h(t_j)^2+V(t_j).
 \end{split}
\end{equation}
We have $h=h'=0$ for $t\le\delta_g$, $h(t)=t$ for $t\ge2\delta_g$, $h,h'\ge0$, and $h(t)\le\max\{t,0\}$. Moreover
\begin{equation}\label{hd:eq:V-properties}
 -\frac9{10}(c_R+1)\le V\le0,\quad V'\le0,\quad
 V'=0\ (t\le1/10),\quad V'=-(c_R+1)\ (1/5\le t\le1).
\end{equation}
In particular, every source-dependent part of $P_j$ has an exactly constant first jet for $|t_j|\le\delta_g$; the remaining state quadratic is public and isotropic.

Define the outer couplings
\begin{align}
 G_{\mathrm{en},j}
  &=-4q(t_{j-1})(1-q(t_j))\ip{v_{j-1}}{E_\nu(a_j)},\label{hd:eq:global-en}\\
 G_{\mathrm{ex},j}
  &=-q(t_{j-1})(1-q(t_j))\ip{Z(S_j)}{E_\nu(b_j)},\label{hd:eq:global-ex}\\
 G_{\mathrm{st},j}
  &=K_s q(t_j)(1-e_s(t_{j-1}))(1-q(t_{j-1})).\label{hd:eq:global-st}
\end{align}
At the first stage use the public values $t_0=1$, $v_0=e$. There is no variable $S_0$.
The exit \eqref{hd:eq:global-ex} is a \emph{public vector inner product}. It contains no rank-one projection onto $v_j$. Its only source-specific dependence is through the exactly flat state gates. This difference from the ungated rank-one exit is essential.

Using the two-port components \eqref{hd:eq:pipe-components} on each stage, define
\begin{equation}\label{hd:eq:global-components}
 \bar f_i(S,\nu;w)=\sum_{j=1}^M
 \left[\mathscr H_i^{(j)}(a_j,b_j;w^{(j)})+
       G_{\mathrm{en},j}+G_{\mathrm{ex},j}+G_{\mathrm{st},j}+P_j(S_j)\right],
 \qquad \nu=((a_j,b_j))_j.
\end{equation}
The known dual domain is the single Euclidean ball
\begin{equation}\label{hd:eq:global-dual-domain}
 Y=\{w:\norm w\le20\sqrt3\,n^{1/4}N\}.
\end{equation}
Let $\bar\Phi(S,\nu)=\max_{w\in Y}n^{-1}\sum_i\bar f_i(S,\nu;w)$.

\begin{lemma}[Bounded-overlap Hessian summation]\label{hd:lem:overlap}
Let $G=\sum_{r}G_r$ be a twice differentiable function on an orthogonal block product.  Suppose each $G_r$ depends on a block subset $I_r$, $\|\nabla^2G_r\|_{\rm op}\le C_r$, and every block belongs to at most $d_0$ subsets.  If $C_r\le C$, then
\[
 \|\nabla^2G\|_{\rm op}\le d_0C.
\]
More generally the right side is the maximum, over blocks, of the sum of the incident $C_r$'s.
\end{lemma}
\begin{proof}
For any block vector $h$, with $P_rh$ its restriction to $I_r$,
\[
 |\langle h,\nabla^2G h\rangle|
 \le\sum_r C_r\|P_rh\|^2.
\]
Each block norm is counted only by its incident terms.  Taking the maximum incident sum gives the assertion.  This is the precise reason that the global smoothness constant below does not acquire an $M$ or $\sqrt M$ factor.
\end{proof}

\begin{theorem}[Global admissibility and initial gap]\label{hd:thm:global-regularity}
For all choices of the fixed rotations, every component \eqref{hd:eq:global-components} is globally jointly $L_0$-smooth with the numerical choice $L_0=10^8$. The average is $n^{-1/2}N^{-2}$-strongly concave in $w$. The value function is differentiable, $L_0$-weakly convex and bounded below, and
\begin{equation}\label{hd:eq:global-gap}
 \bar\Phi(0)=0,\qquad \bar\Phi(0)-\inf\bar\Phi\le C_\Delta M,
 \qquad C_\Delta=31000.
\end{equation}
\end{theorem}
\begin{proof}
Strong dual concavity and the component Hessian bound of eight for the two-port part have already been proved. Those components are block diagonal across stages.
We give generous explicit Hessian bounds for the new terms. For a state score, $\norm{Dt_j}\le1$, $\norm{D^2t_j}\le2$. The composed gate has derivative norm at most ten and Hessian norm at most $220$; the composed $e_s$ has derivative norm at most one and Hessian norm at most $18$. The connector cap has derivative norm at most one and Hessian norm at most one, with image norm at most $42$. The state cap has image norm at most $20$ and Hessian norm at most two.

On the score range $[-20,20]$, $|(h^2)'|\le680$ and $|(h^2)''|\le1666$. For the latter, in the only transition interval $[\delta_g,2\delta_g]$, one has $h'\le17$, $|h''|\le272/\delta_g$, and $h\le2\delta_g$; outside it the squared function is either zero or $t^2$. Also $|V'|\le32417$ and $|V''|\le160\cdot32417$. The chain rule bounds the Hessian of $P_j$ by
\[
 4300+2150(1666+2\cdot680)+160\cdot32417+2\cdot32417
 =11761754<12\cdot10^6.
\]

The product rule bounds the Hessians of $G_{\mathrm{en},j}$, $G_{\mathrm{ex},j}$ and $G_{\mathrm{st},j}$ respectively by $2\cdot10^5$, $6\cdot10^5$ and $4\cdot10^5$. To check the largest product bound, the inner product in the exit has absolute value at most $840$, gradient norm at most $62$, and Hessian norm at most $106$. Multiplying it by two gates with derivative/Hessian bounds $10,220$ gives
\[
 2\cdot220\cdot840+2\cdot10^2\cdot840+4\cdot10\cdot62+106
 =540186<6\cdot10^5,
\]
also when one gate and the inner product share the state variable.
For the state-coupling product, its preceding-state factor has value,
gradient, and Hessian bounds $3.5,36,808$. Its Hessian is therefore
at most $172(220\cdot3.5+808+2\cdot10\cdot36)=395256<4\cdot10^5$.
The only overlaps are local: a state occurs in its penalty and in the two neighboring stages; connectors occur in one outer stage. Bounding the quadratic form of each local Hessian on its support, then summing, gives a full outer Hessian norm below $4\cdot10^7$. Adding the two-port part is covered by $10^8$. These estimates use full Euclidean operator norms and are independent of all block dimensions and of $M,N,n$.

Since $|h(t_j)|\le|t_j|\le\norm{S_j}$, the first two terms of $P_j$ sum to a nonnegative number. At $w=0$, the two-port part is nonnegative. Also $G_{\mathrm{st},j}\ge0$: where $q(t_{j-1})<1$, monotonicity gives $e_s(t_{j-1})\le1$; elsewhere the last factor is zero. The entrance and exit are bounded below by $-4\sigma$ and $-2R_s\sigma$. Thus feasibility of $w=0$ gives
\[
 \bar\Phi\ge-M\left[\frac9{10}(c_R+1)+4\sigma+2R_s\sigma\right]
 >-31000M.
\]
At the primal origin all outer terms vanish and the average dual quadratic is maximized at zero. This proves \eqref{hd:eq:global-gap}. Uniqueness of the dual maximizer on the compact ball gives differentiability by direct differentiation of the maximum. Finally, adding $L_0\norm{(S,\nu)}^2/2$ makes every fixed-dual average function convex; taking its maximum proves weak convexity.
\end{proof}

\subsection{No low or intermediate state can hide outside the identity region}
Write $G=\sum_j(G_{\mathrm{en},j}+G_{\mathrm{ex},j}+G_{\mathrm{st},j})$.
It is useful to separate the direct occurrence of $Z(S_j)$ in the exit from all occurrences of the score $t_j$. The state gradient has the exact form
\begin{equation}\label{hd:eq:state-gradient-form}
 \nabla_{S_j}\bar\Phi
 =\gamma S_j+B_j\nabla t_j
  -q(t_{j-1})(1-q(t_j))DZ(S_j)^{\mathsf T}E_\nu(b_j),
 \quad B_j=-\gamma h(t_j)h'(t_j)+V'(t_j)+C_j,
\end{equation}
where $C_j$ is the sum of score derivatives of the outer couplings, with the direct vector occurrence held fixed. This formula holds even when the dual ball is active, because the inner value is independent of the states.

The product bounds above imply
\begin{equation}\label{hd:eq:score-derivative-bound}
 |C_j|\le K_s(1+7\cdot10)+8\cdot10\sigma+4R_s\cdot10\sigma=32372.
\end{equation}
There are two entrance, two exit, and two state-coupling contributions; missing boundary contributions are simply omitted. For $t_j\le1$, a weighted bound is
\begin{equation}\label{hd:eq:weighted-score-bound}
 |t_j C_j|\le K_s(3+7\cdot10)+8\cdot10\sigma+4R_s\cdot10\sigma=32716.
\end{equation}
The only change is $|t e_s'(t)|\le3$ in place of $|e_s'|\le1$; every $q'$ term has $t\in[1/5,1]$.
For $t\le1$, one has $|t h(t)h'(t)|\le17$ (the expression is zero for $t\le0$), and $|tV'(t)|\le32417$. Therefore
\begin{equation}\label{hd:eq:B-weighted}
 |t_j B_j|\le4300\cdot17+32417+32716=138233
 \qquad(t_j\le1).
\end{equation}

\begin{lemma}[Localization without any connector or dual-interiority assumption]\label{hd:lem:state-localization}
If $\norm{\nabla\bar\Phi}\le\tau_0$ and $t_j\le1$, then $\norm{S_j}<10$. In particular $Z(S_j)=S_j$ and $t_j=\ip{v_j}{S_j}$ at that state.
\end{lemma}
\begin{proof}
Suppose $r=\norm{S_j}\ge10$, and let $f=f_{10}(r)$. In the radial direction $S_j/r$,
\[
 D t_j[S_j/r]=f_{10}'(r)t_j/f,
 \qquad DZ(S_j)[S_j/r]=f_{10}'(r)S_j/r.
\]
Use $f\ge10$, $0\le f_{10}'\le1$, \eqref{hd:eq:state-gradient-form}, \eqref{hd:eq:B-weighted}, and $\norm{E_\nu(b_j)}\le42$. The radial derivative is at least
\[
 4300r-138233/10-42\ge29134.7>1.
\]
This contradicts the small-gradient assumption. Notice that no bound on $b_j$, on any other state, or on the unconstrained dual maximizer was used.
\end{proof}

\subsection{A single frontier and a constant total connector norm}
\begin{lemma}[State classification and exclusion of a low--high inversion]\label{hd:lem:state-classification}
At every point with $\norm{\nabla\bar\Phi}\le\tau_0$, each state is either
\[
 \text{low: }-1/10<t_j\le1/5,
 \qquad\text{or high: }t_j\ge1.
\]
A low state cannot be followed by a high state.
\end{lemma}
\begin{proof}
For all scores at most one, Lemma~\ref{hd:lem:state-localization} puts the state inside the identity region. We may take its directional derivative along the unit vector $v_j$ as an ordinary scalar score derivative, except for the direct vector exit.
If $t_j\le-1/10$, then $q=q'=h=V'=0$ at that score. The only nonzero score derivative among the outer terms is the outgoing state coupling, which is nonpositive because $e_s'\ge0$. The direct exit derivative is at most $\sigma$. Hence
\[
 \ip{v_j}{\nabla_{S_j}\bar\Phi}\le4300t_j+42\le-388<-1.
\]
If $1/5<t_j<1$, then $h(t_j)=t_j$, so the isotropic derivative $\gamma t_j$ cancels its source correction. Equations \eqref{hd:eq:V-properties} and \eqref{hd:eq:score-derivative-bound} give
\[
 \ip{v_j}{\nabla_{S_j}\bar\Phi}
 \le-32417+32372+42=-3.
\]
Both cases are impossible.

For a low state define its net scalar penalty derivative
\[
 D(t)=\gamma t-\gamma h(t)h'(t)+V'(t).
\]
It obeys $D(t)\le2\gamma\delta_g=1/50$: for $t\le0$ it is nonpositive; for $0<t<2\delta_g$, use $hh'\ge0$; for $t\ge2\delta_g$, the two quadratic derivatives cancel and $V'\le0$.
If the next state is high, its gate equals one, whereas the low state's gate and its derivative are zero. Its outgoing state-coupling derivative is exactly $-K_s=-172$, since $e_s'=1$ on the low interval. All other score derivatives vanish; the direct exit contributes at most $42$. Thus
\[
 \ip{v_j}{\nabla_{S_j}\bar\Phi}\le1/50+42-172<-1,
\]
again impossible.
\end{proof}

\begin{lemma}[Connector control at the only possible frontier]\label{hd:lem:global-connectors}
Suppose $\norm{\nabla\bar\Phi}\le\tau_0$ and $t_M\le1/5$. There is a unique $j$ such that states before $j$ are high and states $j,\ldots,M$ are low. Moreover
\begin{equation}\label{hd:eq:connector-final-bound}
 \norm{S_j}<1/4,\qquad \norm\nu\le43/4<21.
\end{equation}
These conclusions hold for the actual value function on the single dual ball, before asserting that its unconstrained dual maximizer is feasible.
\end{lemma}
\begin{proof}
State classification, absence of a low--high inversion, and the public high root give the unique frontier. Since $q(t_{i-1})(1-q(t_i))$ vanishes except at $i=j$, only that stage contributes to $\nabla_\nu G$.
At its low state the cap is the identity. There is no state-coupling derivative, since all later gates are zero; its own gate derivative is zero. Projecting \eqref{hd:eq:state-gradient-form} onto $v_j^\perp$ gives
\[
 P_{v_j^\perp}\nabla_{S_j}\bar\Phi
 =4300P_{v_j^\perp}S_j-P_{v_j^\perp}E_\nu(b_j).
\]
Consequently $\norm{P_{v_j^\perp}S_j}\le(42+1/20)/4300<1/100$.
Together with $|t_j|\le1/5$, this implies $\norm{S_j}<1/4$.

The two surviving connector gradients have norms at most four and $\norm{Z(S_j)}<1/4$, respectively. Since the state penalties are independent of $\nu$, Lemma~\ref{hd:lem:bounded-dual} now gives
\[
 \norm\nu\le\frac52(1/20+4+1/4)=43/4.
\]
In particular all connector caps are now in their identity region. The order of the reasoning avoids any assumption about the unconstrained dual solution.
\end{proof}

\begin{theorem}[Global terminal-gradient certificate]\label{hd:thm:global-certificate}
For the explicitly defined functions and the dual ball \eqref{hd:eq:global-dual-domain},
\begin{equation}\label{hd:eq:global-certificate}
 \norm{\nabla\bar\Phi(S,\nu)}\le1/20\quad\Longrightarrow\quad t_M\ge1.
\end{equation}
This includes all off-source, arbitrary-norm primal points and all cases in which the constrained dual maximizer initially might have been on the boundary.
\end{theorem}
\begin{proof}
If the conclusion fails, state classification gives $t_M\le1/5$. Apply Lemma~\ref{hd:lem:global-connectors}. The unconstrained maximizers of all two-port modules satisfy, by their squared radius estimates,
\[
 \norm{w^*(\nu)}\le\sqrt3 Nn^{1/4}\norm\nu
 \le\frac{43}{4}\sqrt3 Nn^{1/4}
 <20\sqrt3 Nn^{1/4}.
\]
Thus they are strictly inside the prescribed \emph{single} dual ball. The actual value and gradient at this point equal the unconstrained eliminated formulas.

Let $P=U_{jm-1}\cdots U_{(j-1)m}$ be the stage transfer. Then $Pv_{j-1}=v_j$ and $\mathcal D_Nv_j=v_j$. At the frontier all caps on connectors and on $S_j$ are identities. Put
\[
 A_j=\ip{v_{j-1}}{a_j},\quad B_j^{\rm sc}=\ip{v_j}{b_j},\quad s=t_j,
 \quad g_a=\ip{v_{j-1}}{\nabla_{a_j}\bar\Phi},\quad
 g_b=\ip{v_j}{\nabla_{b_j}\bar\Phi}.
\]
The two projected connector equations are exactly
\[
 g_a=2A_j-B_j^{\rm sc}-4,\qquad
 g_b=-A_j+2B_j^{\rm sc}-s.
\]
Therefore
\[
 B_j^{\rm sc}=\frac{4+2s+g_a+2g_b}{3}
 >\frac{4-1/5-3/20}{3}=\frac{73}{60}.
\]
The state derivative along $v_j$ is $D(s)-B_j^{\rm sc}$, with $D(s)\le1/50$ as proved above. It is less than $1/50-73/60<-1$, contradicting $\norm{\nabla\bar\Phi}\le1/20$.
\end{proof}

\begin{remark}[Sanity check: the target is not vacuous]
With every state equal to $2v_j$, all connectors and all dual variables equal to zero, all gates are one, all caps are in their identity regions, and every gradient is zero. Thus the construction has actual stationary points. The certificate says that finding one requires terminal direction information; it does not rely on a function without stationary points.
\end{remark}

\subsection{Moreau outputs and exact source-flat simulation}
\label{hd:sec:flat}
Let $\bar\epsilon=1/40$.  Since $\bar\Phi$ is $L_0$-weakly convex and
bounded below, its proximal objective has a unique minimizer $p$ and
\[
 g=\nabla\bar\Phi_{1/(2L_0)}(x)=2L_0(x-p)\in\partial\bar\Phi(p).
\]
The value function is differentiable here, so the last vector is
$\nabla\bar\Phi(p)$.  If $\|g\|\le\bar\epsilon$,
Theorem~\ref{hd:thm:global-certificate} gives $t_M(p)\ge1$.  The score is
$1$-Lipschitz, hence
\begin{equation}\label{hd:eq:moreau-terminal}
 t_M(x)\ge1-\bar\epsilon/(2L_0)>1/2.
\end{equation}
The vector
$\widehat v=Z(S_M)/(2R_s)$ is computable from the primal output alone and
has norm at most one.  Equation~\eqref{hd:eq:moreau-terminal} implies
\begin{equation}\label{hd:eq:decoding}
 \langle v_M,\widehat v\rangle>1/(4R_s)=1/40.
\end{equation}
No proximal point is supplied to the algorithm or entered into the
information accounting.

\begin{lemma}[Radial caps create no new orientation]
\label{hd:lem:radial-span}
For every cap $\rho_C$ in \eqref{hd:eq:radial-cap}, every $x,y$ in its
Euclidean space satisfy
\[
 \rho_C(x)\in\operatorname{span}\{x\},\qquad
 D\rho_C(x)^{\mathsf T}y\in\operatorname{span}\{x,y\}.
\]
Thus the values and gradients of all outer terms in
\eqref{hd:eq:global-components}, once their scalar coefficients and any
needed released sources are known, lie in the span of the query blocks and
those released sources.
\end{lemma}

\begin{proof}
For $x\ne0$, write $\rho_C(x)=\alpha(\|x\|)x$.  Then
\[
 D\rho_C(x)^{\mathsf T}y
 =\alpha(\|x\|)y+
 \frac{\alpha'(\|x\|)}{\|x\|}\langle x,y\rangle x.
\]
The formula extends continuously through the origin.  Apply it to
$E_\nu=\rho_{21}$ and $Z=\rho_{10}$ and use the product rule.
\end{proof}

\begin{lemma}[Prefix releases and exact inactive first jets]
\label{hd:lem:source-flat-jet}
Close the release rule after each matrix round by releasing every eligible
macro source.  Then the released macro sites always form a prefix.  Suppose
that, at a component query, every unreleased macro site satisfies
$|t_j|<\delta_g$.  The complete value--gradient first jet of every component
$\bar f_i$ at that query is then exactly computable from
\begin{enumerate}[label=(\alph*),itemsep=2pt]
\item the public query blocks and scalar primitives;
\item the forward and transpose matrix actions in the two-port modules; and
\item the already released source vectors.
\end{enumerate}
No unreleased source vector, projection onto such a source, or derivative in
such a direction is needed.  Moreover every vector in the computed gradient
belongs, at its depth, to the span of the marked query vectors, the matrix
reply vectors, and the released sources.
\end{lemma}

\begin{proof}
Deletion closure gives $p_0\ge p_1\ge\cdots\ge p_H$.  Hence, if the source
at a later macro site $jm$ is eligible, every earlier unreleased macro
source is eligible as well.  Repeating releases until none is eligible makes
the released set a prefix.

Because $\delta_g<1/10<1/5$, every $|t|<\delta_g$ obeys the exact identities
\begin{equation}\label{hd:eq:inactive-jets}
 q(t)=q'(t)=0,
 \qquad h(t)=h'(t)=0,
 \qquad V(t)=V'(t)=0.
\end{equation}
We inspect the source-dependent terms stage by stage; ``zero first jet''
means zero value and zero derivative with respect to every primal block.

If the preceding source of $G_{\rm en,j}$ is unreleased, then
$q(t_{j-1})=q'(t_{j-1})=0$, so the entrance has zero first jet despite the
bare inner product with $v_{j-1}$.  If the preceding source is released, it
is known.  When the current source is unreleased,
$1-q(t_j)=1$ with zero derivative, so no current source is needed.

The same preceding gate makes $G_{\rm ex,j}$ have zero first jet whenever
the preceding source is unreleased.  Otherwise that source is known, and an
unreleased current source occurs only through the constant first jet of
$1-q(t_j)$.  The direct inner product
$\langle Z(S_j),E_\nu(b_j)\rangle$ is public.

If the current source of $G_{\rm st,j}$ is unreleased, then
$q(t_j)=q'(t_j)=0$ and the entire state-coupling first jet is zero.  If it is
released, the prefix property says that the preceding source is also
released (with the public convention $t_0=1$ at the first stage), so the
term is exactly evaluable.  This is the only place where the nonflat factor
$e_s(t_{j-1})$ could otherwise expose an earlier source.

Finally, at an unreleased current site, \eqref{hd:eq:inactive-jets} reduces
$P_j$ and its gradient exactly to
$\gamma\|S_j\|^2/2$ and $\gamma S_j$.  At a released site the source is
known and the full term is evaluated.  The two-port quadratic contains no
source vector and its complete first jet is obtained from the stated matrix
actions and public arithmetic.  These cases exhaust
\eqref{hd:eq:global-components}.

The final span assertion follows from
Lemma~\ref{hd:lem:radial-span}: cap values are scalar multiples of their
inputs, cap adjoints applied to another vector lie in the span of the two,
and all remaining operations are scalar linear combinations.  Hence the
reply itself exposes no orientation beyond the marked inputs, matrix
outputs, and released sources already present in the lazy coupling.
\end{proof}

\begin{proposition}[Static transcript equality on one common event]
\label{hd:prop:static-transcript}
Fix a deterministic component-oracle algorithm using at most $K$ queries.
Run the following simulator against the enlarged branch oracle.
For each component query, use the same selected component index on every
edge, mark every physical coordinate block occurring at each depth, obtain
the required forward and transpose actions, close all eligible releases,
and return the first jet described in
Lemma~\ref{hd:lem:source-flat-jet}.  At termination mark
$Z(S_M)/(2R_s)$.

At every depth there are at most five marked query vectors, so this is an
$8$-channel strategy with at most eight terminal slots.  Apply
Lemma~\ref{hd:lem:lazy} with $b=b_{\rm out}=8$, and put
\[
 a_{\rm flat}=\frac{\delta_g}{2R_s},\qquad
 \rho=\frac{a_{\rm flat}^2}{2},\qquad
 \mathcal E=\{\|P_Ve\|^2\le\rho\}.
\]
After completing the partial rotations, their marginal law is the fixed
product--Haar law defining the static components
\eqref{hd:eq:global-components}.  On $\mathcal E$, the simulator's entire
transcript and terminal output coincide, query by query, with the genuine
value--full-gradient transcript and output of that static finite-sum
instance.
\end{proposition}

\begin{proof}
At an interior branch depth a query has one dual vector.  At a shared macro
endpoint there can be two distinct endpoint dual vectors, the two adjacent
connectors, and one state vector, for a total of five.  Boundary depths have
fewer.  The cap images require no additional slots because they are scalar
multiples of the corresponding query vectors.  Thus eight slots are more
than sufficient.  The selected component $i$ requires exactly one forward
and one transpose action on each internal edge, as in
Theorem~\ref{hd:thm:pipe}(iv).

We prove transcript equality by induction.  Before the first query the two
histories agree.  If they agree through query $k-1$, determinism makes the
$k$th component index and physical query point identical.  The lazy
coupling processes all of its local vectors before any reply and returns the
exact static matrix actions.

After release closure, any unreleased macro site $j$ has
$p_{jm}<\rho$.  The vector $Z(S_j)/(2R_s)$ has norm at most one.  Since
$S_j$ was marked and $Z(S_j)$ is its scalar multiple,
Lemma~\ref{hd:lem:support} applies.  On $\mathcal E$,
\[
 \frac{|t_j|^2}{4R_s^2}
 \le p_{jm}+\|P_Ve\|^2
 <2\rho=a_{\rm flat}^2,
\]
so $|t_j|<\delta_g$.  Lemma~\ref{hd:lem:source-flat-jet} therefore says
that the simulator returns exactly the value and every gradient block of
$\bar f_i$ at this query.  By its span conclusion, recording this reply does
not require any additional unaccounted orientation exposure.  There can be
no first mismatching reply, and the induction continues through all
queries.  The deterministic terminal output is consequently identical as
well.  Completing the rotations has the static product--Haar marginal by
Lemma~\ref{hd:lem:lazy}.
\end{proof}

\begin{theorem}[Fixed-law component lower bound for the global assembly]
\label{hd:thm:unscaled-lower}
There is a numerical $c_*>0$ such that every unrestricted randomized
component algorithm succeeding with probability $2/3$ at Moreau threshold
$\bar\epsilon$ on every member of \eqref{hd:eq:global-components} requires
more than $c_*nM(N-1)$ calls, up to integer rounding.  The hard distribution
is static, is independent of the algorithm's private seed, and is supported
on finite-dimensional instances fixed before the interaction.
\end{theorem}

\begin{proof}
Put $H=M(N-1)$ and use macro source sites $jm$.  Let
\[
 a_*=\frac{\delta_g}{2R_s}=\frac1{8600000},
 \qquad \rho=a_*^2/2,
 \qquad
 K_*=\left\lfloor\frac{\rho^2nH}{4\mathrm e}\right\rfloor,
\]
and choose
\begin{equation}\label{hd:eq:global-coupling-dimension}
 D=\left\lceil
 \frac{32(8K_*+8)}{a_*^2}\sum_{t=0}^Hn^t
 \right\rceil.
\end{equation}
This fixes all primal and dual dimensions before the algorithm runs.  The
full optimization dimension is finite; for example it is at most a
universal constant times $D\sum_{t=0}^Hn^t$.

Fix first a deterministic algorithm using $K\le K_*$ component queries and
form the simulator in Proposition~\ref{hd:prop:static-transcript}.  It is a
valid enlarged-oracle strategy: extra marked channels reveal nothing, and
one component query uses only the single selected forward and transpose
action on each edge.  All other operations are public arithmetic or use a
source supplied by the release rule.

The reservoir in \eqref{hd:eq:global-coupling-dimension} satisfies
\[
 \mathbb E\|P_Ve\|^2\le a_*^2/32,
 \qquad
 \mathbb P(\mathcal E^c)\le
 \frac{a_*^2/32}{\rho}=\frac1{16}.
\]
No Haar assertion is conditioned on $\mathcal E$; the event is imposed only
after the unconditional lazy construction has produced the static
product--Haar marginal.

On $\mathcal E$, Proposition~\ref{hd:prop:static-transcript} identifies the
simulated run with the genuine static component-oracle run.  The accounting
theorem, with $K\le K_*$ and at most $M\le H$ releases, gives
$p_H\le\rho/4$ exactly as in Theorem~\ref{hd:thm:serial}.  Applying
Lemma~\ref{hd:lem:support} to the marked terminal vector
$\widehat v=Z(S_M)/(2R_s)$ yields
\[
 |\langle u_H,\widehat v\rangle|^2
 \le p_H+\|P_Ve\|^2
 \le\rho/4+\rho
 =5a_*^2/8<a_*^2.
\]
In particular the terminal correlation is below $a_*$.  But if the genuine
static run produced a $\bar\epsilon$-Moreau-stationary point,
\eqref{hd:eq:decoding} would give correlation greater than $1/40>a_*$.  No
successful output is therefore possible on $\mathcal E$, and the success
probability under the fixed product--Haar instance distribution is at most
$1/16$.

For a randomized algorithm, condition on its private seed.  The preceding
argument applies to every seed with the same product--Haar marginal and the
same dimension, so integrating over the seed still gives average success at
most $1/16$.  If the algorithm succeeded with probability at least $2/3$
on every fixed instance, its average success under this static distribution
would be at least $2/3$, a contradiction.  Thus, for every randomized
algorithm with budget $K\le K_*$, some fixed instance in the support violates
the required success guarantee.  Since an integer successful budget must be
strictly larger than
$K_*=\lfloor\rho^2nH/(4\mathrm e)\rfloor$, one may take
$c_*=\rho^2/(4\mathrm e)$ with no hidden additive loss.
\end{proof}

\subsection{Parameter scaling and the matching region}\label{hd:sec:scaling}
We conclude the proof of Theorem~\ref{hd:thm:main}. Let $C_Y=20\sqrt3$, $L_0=10^8$, $C_\Delta=31000$, and $\bar\epsilon=1/40$ as above. For prescribed $L,\epsilon$, choose the spatial scale
\[
 \lambda=\frac{L_0\epsilon}{L\bar\epsilon},\qquad
 f_i(x,y)=\frac L{L_0}\lambda^2\bar f_i(x/\lambda,y/\lambda).
\]
The scaled dual ball has diameter $2\lambda C_Y n^{1/4}N$. The component smoothness bound is $L$, and the Moreau convention scales exactly:
\begin{equation}\label{hd:eq:moreau-scaling}
 \nabla\Phi_{1/(2L)}(x)
 =\frac{L\lambda}{L_0}\nabla\bar\Phi_{1/(2L_0)}(x/\lambda)
 =\frac\epsilon{\bar\epsilon}
      \nabla\bar\Phi_{1/(2L_0)}(x/\lambda).
\end{equation}
This follows by changing variables in the proximal minimization, whose quadratic coefficient becomes $L_0$ inside the unscaled bracket.

Define numerical constants
\[
 \eta_M=\frac{\bar\epsilon^2}{L_0C_\Delta},\qquad
 \eta_N=\frac{\bar\epsilon}{2L_0C_Y},\qquad
 A_0=\frac2{\eta_M},\qquad B_0=\frac4{\eta_N}.
\]
For $A\ge A_0$ and $B\ge B_0n^{1/4}$ take
\[
 M=\lfloor \eta_M A\rfloor,\qquad
 N=\lfloor \eta_N B/n^{1/4}\rfloor.
\]
Then $M\ge1$, $N\ge2$, the gap is at most $\Delta$, and the dual diameter is at most $D_Y$. Moreover
\[
 M\ge \eta_M A/2,\qquad N-1\ge \eta_N B/(2n^{1/4}).
\]
Theorem~\ref{hd:thm:unscaled-lower} and the exact scaled oracle simulation imply
\[
 \T>c_*nM(N-1)\ge\frac{c_*\eta_M\eta_N}{4}\,n^{3/4}AB,
\]
with integer rounding interpreted by requiring a successful integer budget to exceed the corresponding floor. The upper bound \eqref{eq:accelerated-main} reduces to the right side of \eqref{hd:eq:matching} because $A\ge A_0$ and $B\ge B_0n^{1/4}$ absorb the repeated $n$ initialization cost. No amortized initialization theorem is needed in this region.

\section{Lower-Bound Reductions and Completion of the Main Proof}
\label{app:reductions}
\subsection{The smooth finite-sum baseline}
\label{sec:moreau-facts}
We use Theorem~3.1 of Peng, Tang, and Jia~\cite{peng}, with the
value-and-gradient oracle and parameter conventions in their
Appendix~G.1. That theorem permits arbitrary adaptive component and
point queries, private randomness, and unqueried outputs. Its
constants can be chosen with $n_0=1024$ and failure threshold
$p_0=11/16$, so failure probability $1/3$ is covered. We import this
preprint theorem as a lower-bound input; the reduction
below is proved here.

\begin{lemma}[Smooth baseline in the original Moreau convention]
\label{lem:baseline}
There are universal $A_{\mathrm{base}},c>0$ such that for every
$n\ge1$, $D_Y\ge0$, and $A\ge A_{\mathrm{base}}$,
\[
 \T\ge c(n+\sqrt n A).
\]
The hard components can be chosen independently of the dual variable.
\end{lemma}
\begin{proof}
First, the additive $n$ term has a direct construction. Let $J$ be
uniform in $[n]$, let $s$ be an independent uniform sign, and set
\[
 F_i(x)=\frac L2x^2-nLsR\,\mathbf1_{\{i=J\}}x,
 \qquad R=\sqrt{2\Delta/L}.
\]
Each component is individually $L$-smooth. The average has minimizer
$sR$, gap $\Delta$ at zero, and Moreau gradient $(2L/3)(x-sR)$.
For $A\ge2$, the two successful output intervals are disjoint.
Before the algorithm queries $J$, its transcript is independent of
$s$. For $Q\le n$, success is at most $1/2+Q/(2n)$, so a uniform
$2/3$ guarantee requires $Q\ge n/3$. All choices of $x$ and all
component value queries are included in this argument.

Let $T$ be the fixed budget of an alleged algorithm with the required
worst-case success guarantee. Let $N=1024$ and let
$c_0,c_{\mathrm{ext}}>0$ denote universal
accuracy and complexity constants from the imported theorem.
Take $A_{\mathrm{base}}\ge\max\{2,9/(4c_0)\}$.
For $n\ge N$, embed any individually $L$-smooth finite sum
$F=n^{-1}\sum_iF_i$ as $f_i(x,y)=F_i(x)$ on a fixed known compact
dual domain of permitted diameter. One component query is simulated
by one original query. By \cref{lem:transfer}, a successful Moreau
output is a direct-gradient output with tolerance
$\eta=3\varepsilon/2$. The imported accuracy requirement
$\eta^2\le c_0L\Delta$ holds, and the success probability is
unchanged. Hence
\[
 T\ge c_{\mathrm{ext}}
       \left(n+\frac49\sqrt n A\right).
\]
No Moreau-gradient evaluation or verification pass is required.

For $1\le n<N$, start from an $N$-component external instance
$F_0=N^{-1}\sum_{j=1}^N G_j$ and make each of the $n$ target
components equal to $F_0$. Each target component is $L$-smooth and
the average, initial gap, and value function are unchanged. Simulate
one target query by all $N$ external component queries. The same
stationarity transfer yields
\[
 NT\ge c_{\mathrm{ext}}
       \left(N+\frac49\sqrt N A\right).
\]
Since $n\le N$ and $N$ is universal, this implies
$T\ge (4c_{\mathrm{ext}}/(9N))(n+\sqrt n A)$.
Reducing the universal constant covers all component counts.
\end{proof}

\subsection{Balanced replication with the exact Moreau parameter}
\label{repl:section}
The following reduction preserves the average objective for every
integer component count. The factor of two reserved in individual
smoothness handles unequal group sizes, and
\cref{repl:moreau-comparison} accounts for the Moreau-parameter change.

\begin{lemma}[Balanced replication and one-query simulation]
\label{repl:balanced-replication}
Let $1\le s\le n$ be integers.  Let $g_1,\ldots,g_s$ be components
with jointly $L/2$-Lipschitz gradients on $\mathbb R^{d_x}\times Y$,
whose average is concave in $y$.  Partition $\{1,\ldots,n\}$ into
public groups $I_1,\ldots,I_s$ with sizes
\[
 r_j=|I_j|\in\{\lfloor n/s\rfloor,\lceil n/s\rceil\},
 \qquad\sum_{j=1}^sr_j=n,
\]
and define
\begin{equation}
 f_i(x,y)=\frac{n}{s r_j}g_j(x,y),\qquad i\in I_j.
 \label{repl:component-definition}
\end{equation}
Then every $f_i$ has jointly $L$-Lipschitz gradient, the two average
functions and value functions are identical, and one value--full-gradient
query to any $f_i$ can be simulated by one query to its corresponding
$g_j$.  The known dual domain and the initial value gap are unchanged.

Moreover, if an algorithm for the $n$-component instance outputs $x$
with
\[
 \|\nabla\Phi_{1/(2L)}(x)\|\le\epsilon,
\]
then the same output for the $s$-component instance satisfies
\[
 \|\nabla\Phi_{1/(2(L/2))}(x)\|\le3\epsilon/2.
\]
The reduction preserves every fixed query budget, adaptive query choice,
private random seed, and success probability.
\end{lemma}
\begin{proof}
Since $n/s\ge1$,
\[
 0<\frac{n}{sr_j}
 \le\frac{n/s}{\lfloor n/s\rfloor}\le2.
\]
Positive rescaling therefore gives the asserted component smoothness.
Summing within the groups shows
\begin{equation}
 \frac1n\sum_{i=1}^nf_i(x,y)
 =\frac1n\sum_{j=1}^sr_j\frac{n}{sr_j}g_j(x,y)
 =\frac1s\sum_{j=1}^sg_j(x,y).
 \label{repl:average-identity}
\end{equation}
Thus average dual concavity, the value function $\Phi$, its gap, and the
dual domain are preserved exactly.  For a selected $i\in I_j$, query
$g_j$ at the same $(x,y)$ and multiply its value and both gradients by
$n/(sr_j)$.  This returns the exact oracle reply for $f_i$.

The shared value function is actually $L/2$-weakly convex: each fixed-dual
average has gradient Lipschitz constant at most $L/2$ in $x$, and taking
the maximum preserves weak convexity.  Lemma~\ref{repl:moreau-comparison}
therefore applies with $\rho=L/2$.  The stated implication follows from
\eqref{repl:factor-three-halves}.  Running the original algorithm on
these exact replies, with the same seed, proves all transcript and
budget assertions, including unrestricted final outputs.
\end{proof}

\begin{theorem}[Replicated lower envelope]
\label{repl:lower-envelope}
By \cref{hd:thm:main,hd:sec:scaling}, the high-diameter theorem
holds with numerical constants
$A_{\rm hd},B_{\rm hd},c_{\rm hd}>0$: for every positive integer $s$,
every nominal smoothness $L_g>0$, and every accuracy $\epsilon_g>0$,
the original-model minimax complexity obeys
\begin{equation}
 \begin{gathered}
 \T(s,L_g,\Delta,D_Y,\epsilon_g)
 \ge c_{\rm hd} A_gs^{3/4}B_g,\\
 A_g=\frac{L_g\Delta}{\epsilon_g^2}\ge A_{\rm hd},
 \quad
 B_g=\frac{L_gD_Y}{\epsilon_g}\ge B_{\rm hd}s^{1/4}.
 \end{gathered}
 \label{repl:hd-input}
\end{equation}
Increase $B_{\rm hd}$ to at least one if necessary.  Also, \cref{lem:baseline} gives the smooth finite-sum benchmark in
the same unrestricted oracle and Moreau convention, with numerical constants $A_{\rm sm},c_{\rm sm}>0$ such that
\begin{equation}
 \T(n,L,\Delta,D_Y,\epsilon)
 \ge c_{\rm sm}(n+\sqrt n\,A),
 \qquad A=\frac{L\Delta}{\epsilon^2}\ge A_{\rm sm}.
 \label{repl:smooth-input}
\end{equation}
This benchmark may be realized by components independent of $y$.
Then there are numerical $A_{\rm rep},c_{\rm rep}>0$ such that, for all $n\ge1$,
$B=LD_Y/\epsilon\ge0$, and $A\ge A_{\rm rep}$,
\begin{equation}
 \T(n,L,\Delta,D_Y,\epsilon)
 \ge c_{\rm rep}\left[
 n+A\min\{B^4,n^{3/4}B\}\right].
 \label{repl:final-envelope}
\end{equation}
One may take $A_{\rm rep}=\max\{\tfrac92 A_{\rm hd},A_{\rm sm}\}$.
\end{theorem}
\begin{proof}
Set $b_*=3B_{\rm hd}\ge3$.  First suppose $B\ge b_*$.  Define
\begin{equation}
 X=\min\{n,(B/b_*)^4\},\qquad s=\lfloor X\rfloor.
 \label{repl:effective-count}
\end{equation}
Here $X\ge1$, so $1\le s\le n$ and $s\ge X/2$.  Apply the
high-diameter theorem at
\begin{equation}
 L_g=L/2,\qquad\epsilon_g=3\epsilon/2,
 \qquad A_g=2A/9,\qquad B_g=B/3.
 \label{repl:scaled-input}
\end{equation}
The accuracy condition holds by the definition of $A_{\rm rep}$, and
$s\le(B/b_*)^4$ implies
$B_g\ge B_{\rm hd}s^{1/4}$.  Balanced replication, including its
Moreau-gradient transfer, therefore yields
\begin{equation}
 \T(n,L,\Delta,D_Y,\epsilon)
 \ge\frac{2c_{\rm hd}}{27}A B s^{3/4}.
 \label{repl:transferred-bound}
\end{equation}
The rounding and comparison of the two branches are explicit:
\begin{align}
 B s^{3/4}
 &\ge2^{-3/4}\min\{n^{3/4}B,B^4/b_*^3\}\notag\\
 &\ge2^{-3/4}b_*^{-3}\min\{n^{3/4}B,B^4\}.
 \label{repl:rounding-bound}
\end{align}
Thus the nonlinear term in \eqref{repl:final-envelope} has lower-bound
coefficient
\[
 d_*=\frac{2c_{\rm hd}}{27\,2^{3/4}b_*^3}>0.
\]
The independent benchmark \eqref{repl:smooth-input} supplies
$\T\ge c_{\rm sm}n$.  Combining the two minimax
lower bounds by $\max\{u,v\}\ge(u+v)/2$ gives
\[
 \T
 \ge\tfrac12\min\{c_{\rm sm},d_*\}
       [n+A\min\{B^4,n^{3/4}B\}]
 \qquad(B\ge b_*).
\]
The two hard families need not be combined into a single instance to
justify this maximum.

If instead $0\le B<b_*$, write
$F_B=\min\{B^4,n^{3/4}B\}$.  Since $n\ge1$,
\[
 F_B\le B^4\le b_*^4\sqrt n,
 \qquad
 n+AF_B\le b_*^4(n+A\sqrt n).
\]
The benchmark alone proves the desired lower bound with coefficient
$c_{\rm sm}/b_*^4$.  Consequently
\[
 c_{\rm rep}
 =\tfrac12\min\{c_{\rm sm}/b_*^4,d_*\}
\]
works in both cases.  All floors are valid also for $n=1$ and for
$X=1$; there is no unhandled finite-$n$ or fractional-component case.
All constants are numerical and independent of $n,A,B$.
\end{proof}

\begin{remark}[Exact scope of the reduction]
\label{repl:scope}
The envelope theorem uses precisely the high-diameter lower theorem,
the smooth finite-sum benchmark of \cref{lem:baseline}, and the two
elementary lemmas proved above.  It does not use a new direct-product
assumption.  The small-$B$ benchmark is needed only for a bounded
parameter interval and to supply the additive $n$ term.  The lower bound
is over all admissible known dual domains and finite dimensions; the
replication step itself preserves the chosen dual domain exactly.
\end{remark}

\subsection{Proof of the main characterization}\label{sec:global}
We now prove \cref{global:thm:main}. All construction interfaces are proved above; the smooth baseline is
the explicitly identified external theorem with its proved transfer.

Choose \(A_*\) at least the universal thresholds of the cone theorem,
\cref{hd:thm:main}, the replication theorem, and the baseline transfer.
Thus \(1+A=O(A)\) in the range under consideration.

For \(B\le n^{1/6}\), the cone theorem gives exactly
\(\T=\Theta(n+A\sqrt n\barB)\).
In this range \(\barB^4\le\sqrt n\barB\), while
\(n^{3/4}\barB\ge\sqrt n\barB\); hence
\(\mathcal F(n,B)\asymp\sqrt n\barB\).

For \(n^{1/6}\le B\le n^{1/4}\),
\[
\sqrt n B\le B^4\le n^{3/4}B.
\]
The replicated lower bound is \(\Omega(n+AB^4)\).
The general upper bound \eqref{eq:general-upper} is
\(O(n+(1+A)B^4)\), without a logarithm.
Moreover \(\mathcal F(n,B)\asymp B^4\).

For \(B\ge n^{1/4}\), replication gives
\(\Omega(n+A n^{3/4}B)\) throughout this entire range, including
the constant-factor threshold region of the normalized high-diameter
family. The upper bound \eqref{eq:accelerated-main} is at most
\[
C(1+A)(n+n^{3/4}B)\log(2+n+A+B)
\le C'[n+A n^{3/4}B]\log(2+n+A+B),
\]
because \(n^{3/4}B\ge n\). In this range
\(B^4\ge n^{3/4}B\), so \(\mathcal F(n,B)=n^{3/4}B\).
These comparisons prove both \eqref{global:eq:matching} and
\eqref{global:eq:piecewise}. Since IS($L$) is a subclass of AS($L$)
and both upper bounds hold under AS($L$), the same characterization
holds for $\Tas$ with the stated sandwich inequality.
\end{document}